\documentclass[twoside]{article}

\title{Pathwise Solutions of the 2-D Stochastic Primitive Equations}
\author{Nathan Glatt-Holtz and Roger Temam\\
  \scriptsize{Department of Mathematics and }\\
  \scriptsize{The Institute for Scientific Computing and Applied Mathematics}\\
  \scriptsize{Indiana University, Bloomington}\\
  \scriptsize{emails: negh@indiana.edu, temam@indiana.edu}}
\date{}

\usepackage{amsfonts, amssymb, amsmath, amsthm}
\usepackage[colorlinks=true, pdfstartview=FitV, linkcolor=blue,
            citecolor=blue, urlcolor=blue]{hyperref}
\usepackage[usenames]{color}
\definecolor{Red}{rgb}{0.7,0,0.1}
\definecolor{Green}{rgb}{0,0.7,0}
\usepackage{accents}
\usepackage{comment}

\numberwithin{equation}{section}

\newtheorem{Thm}{Theorem}[section]
\newtheorem{Lem}{Lemma}[section]
\newtheorem{Prop}{Proposition}[section]

\newtheorem{Def}{Definition}[section]
\newtheorem{Rmk}{Remark}[section]

\newcommand{\pd}[1]{\partial_{#1}}
\newcommand{\indFn}[1]{1 \! \! 1_{#1}}
\newcommand{\E}{\mathbb{E}}
\newcommand{\Prb}{\mathbb{P}}
\DeclareMathOperator{\esssup}{ess\,sup}

\includecomment{commentCoAuthor}

\includecomment{commentToDo}

\begin{document}
\markboth{N. Glatt-Holtz and R. Temam} {Pathwise Solutions
of the 2-D Stochastic Primitive Equations}

\maketitle

{\small
\begin{center}
	Dedicated to Alain Bensoussan on the occasion of his 70th birthday.
\end{center}
}

\begin{abstract}
  In this work we consider a stochastic version of the Primitive
  Equations (PEs) of the ocean and the atmosphere and establish the existence and
  uniqueness of pathwise, strong solutions. The analysis employs novel
  techniques in contrast to previous works \cite{EwaldPetcuTemam},
  \cite{GlattHoltzZiane1} in order to handle
  a general class of nonlinear noise structures and to
  allow for physically relevant boundary conditions.  The proof relies
  on Cauchy estimates, stopping time arguments and anisotropic
  estimates.
\end{abstract}

\section{Introduction}
\label{sec:introduction}

The Primitive Equations (PEs) are widely regarded as a fundamental
description of geophysical scale fluid flows. They provide the
analytical core of large General Circulation Models (GCMs) that are
at the forefront of numerical simulations of the earth's ocean and
atmosphere (see e.g. \cite{Trenberth1}). In view of the wide
progress made in computation the need has appeared to better
understand and model some of the uncertainties which are contained
in these GCMs. This is the so called problem of
``parameterization''. Besides all of the physical forms of
parameterization \cite{Trenberth1, ColmanPotter2,ColmanPotter1}
, stochastic modeling has appeared as one of the
major modes in the contemporary evolution of the field
(see \cite{EP09, PenlandSardeshmukh, PenlandEwald1,
Rose1, LeslieQuarini1, MasonThomson,
BernerShuttsLeutbecherPalmer,
ZidikheriFrederiksen}
and also \cite{GlattHoltzTemamTribbia1}). In
this context there is a clear need to better understand the
numerical and analytical underpinnings of stochastic partial
differential equations.

In the present article we will establish the global well-posedness
of the stochastically forced Primitive Equations of the ocean in
dimension two.  While this system has been treated in a simplified
form in previous works, for the case of additive noise
\cite{EwaldPetcuTemam} and nonphysical boundary conditions
\cite{GlattHoltzZiane1}, our aim here is to go further and treat a
more physically realistic version of these equations in the context
of a multiplicative noise. In the formulation herein we face two new
fundamental difficulties in contrast to previous work. Firstly, due
to the imposed boundary conditions we lose higher order cancelations
in the nonlinear terms. This complicates the a priori estimates
which in turn prevent the usage of more direct compactness arguments
adopted in, \cite{Breckner}, \cite{GlattHoltzZiane1}.   On the other
hand, due to the nonlinear multiplicative noise structure, the system may not be
transformed into a random PDE as in \cite{EwaldPetcuTemam}. For this
reason we are not able to treat the probabilistic dependence as a
parameter in the problem. The analysis therefore requires the usage
of advanced tools both from stochastic analysis, namely continuous
time martingale theory and stopping time arguments, and PDE theory
which we treat in detail in a separate work \cite{GlattHoltzTemam1}.

A significant literature exists concerning the
Navier-Stokes equations driven by a multiplicative
volumic white noise forcing.  See 
\cite{BensoussanTemam, Viot1,Cruzeiro1,
CapinskiGatarek, Flandoli1, MikuleviciusRozovskii4,
ZabczykDaPrato2, Breckner,
BensoussanFrehse, BrzezniakPeszat,
MikuleviciusRozovskii2,
GlattHoltzZiane2}.  While our point of view
is similar to some of these works we would
like to point out that the Primitive Equations,
not withstanding very recent results on global
well-posedness in 3D,
are technically more involved than the Navier-
Stokes equations. 

This article is dedicated to Alain Bensoussan
on the occasion of his 70th birthday with
friendship and admiration, and, for the second
author (RT), sweet reminiscences of many
interactions, from Junior High School, to the
early papers on stochastic partial differential
equations \cite{BensoussanTemam2}, 
\cite{BensoussanTemam}, on the subject of this
article, and to many more interactions over
the years.

\subsection{Presentation of the 2D Stochastic PEs}\label{ss1.1a}

 The 2D stochastic Primitive Equations take the form
\begin{subequations}\label{eq:PE2DBasic}
  \begin{gather}
    \pd{t} u + u \pd{x} u + w \pd{z} u - \nu \Delta u -f v + \pd{x} p = F_u + \sigma_u(\mathbf{v},T) \dot{W}_1,
    \label{eq:PEMoment1}\\
    \pd{t} v + u \pd{x} v + w \pd{z} v - \nu \Delta v + f u = F_v + \sigma_v(\mathbf{v},T) \dot{W}_2,
    \label{eq:PEMoment2}\\
    \pd{z}p = - \rho g,
    \label{eq:PEHydroStatic}\\
    \pd{x} u + \pd{z} w = 0,
    \label{eq:PEDivFree}\\
    \pd{t} T + u \pd{x} T + w \pd{z} T - \mu \Delta T = F_T + \sigma_T(\mathbf{v},T) \dot{W}_3,
    \label{eq:PETempCouple}\\
    \rho = \rho_0( 1 - \beta_T ( T - T_0)).
    \label{eq:PEDensityRelationEmp}
  \end{gather}
\end{subequations}
This two dimensional model may be derived from the classical three
dimensional formulation by positing invariance in one of the
horizontal directions, namely the $y-$ (south-north) direction. Here $(\mathbf{v}, w) = (u,v,w)$, $T$, $\rho$
denote respectively the flow field, the temperature and the density of
the fluid being modeled. The coefficients $\nu$, $\mu$ account for the
molecular viscosity and the rate of heat diffusion. A further
parameter $f$, which is a function of the earth's rotation, appears in
an antisymmetric term and is taken constant (see below). The terms
$F_u$, $F_v$ and $F_T$ correspond to external sources of horizontal
momentum and heat. While the first two terms do not usually appear in
practice we retain them here for mathematical generality and to allow
for the possible treatment, not carried out here, of non-homogenous
boundary conditions.

The white noise processes $\dot{W}_i$, the raison
d'\^{e}tre of the present work may be
written in the expansions
\begin{equation}\label{eq:noiseFormalExp}
   \left(
     \begin{split}
       \sigma_{u}(\mathbf{v},T) \dot{W}_1\\
       \sigma_{v}(\mathbf{v},T) \dot{W}_2\\
       \sigma_{T}(\mathbf{v},T) \dot{W}_3\\
     \end{split}
   \right) =
   \sigma_{\mathbf{v},T} (U) \dot{W}
   = \sum_k \sigma^k_{\mathbf{v},T}(U) \dot{W}^k.
\end{equation}
The $\dot{W}^k$s may be interpreted as the time derivatives of a
sequence of independent standard 3-D brownian motions.  However, since
the sample paths of brownian motion are nowhere differentiable we make
rigorous sense of \eqref{eq:PEMoment1}, \eqref{eq:PEMoment2} and
\eqref{eq:PETempCouple} in a time integrated sense, appealing to the
theory of stochastic integration which we consider in the It\={o}
sense.  From the physical point of view these terms may
be introduced in the model as a means to ``parameterize'' physical
and numerical uncertainties.

We consider the evolution of (\ref{eq:PE2DBasic}) over a rectangular
domain $\mathcal{M} = (0,L) \times (-h,0)$ and label the boundary
$\Gamma_i = (0,L) \times \{0\}$, $\Gamma_b = (0,L) \times \{-h\}$ and
$\Gamma_l = \{0, L\} \times (-h,0)$. We posit the physically
realistic boundary conditions
\begin{subequations}\label{eq:PE2DBasicBC}
  \begin{gather}
  \pd{z} \mathbf{v} + \alpha_{\mathbf{v}} \mathbf{v} = 0, \quad
 w = 0, \quad
  \pd{z} T + \alpha_T T = 0, \quad \textrm{ on } \Gamma_i,
  \label{eq:BCTop}\\
  \mathbf{v} = 0, \quad
  \pd{x}T = 0, \quad \textrm{ on } \Gamma_l,
  \label{eq:BCSides}\\
  \mathbf{v} = 0, \quad
  w= 0, \quad
 \pd{z}T = 0, \quad  \textrm{ on } \Gamma_b.\footnotemark
   \label{eq:BCBot}
  \end{gather}
\end{subequations}
\footnotetext{Many of these boundary conditions may be non-homogenous
   (that is that they may include suitable forcing) in general.  We
   will consider only the homogeneous case here.}

The equations and boundary conditions \eqref{eq:PE2DBasic},
\eqref{eq:PE2DBasicBC} are supplemented by initial conditions for $u$,
$v$ and $T$, that is
\begin{equation}\label{eq:basicInitialCond}
   u = u_0, \quad
   v = v_0, \quad
   T = T_0, \quad
   \textrm{ at } t = 0.
\end{equation}

The Primitive equations may be derived from the compressible
Navier-Stokes equations with a combination of empirical observation
and scale analysis. In particular, since deviations of the density of
the fluid from a mean value are small at geophysical scales, the so---called
Boussinesq approximation justifies treating the flow as
incompressible.\footnote{The Boussinesq approximation concerns the oceans.  For the atmosphere we arrive at very
similar equations by considering the pressure as the vertical coordinate, but, for the sake of simplicity, the
emphasis here will be on the case of the oceans.} Another crucial feature, that the ocean and atmosphere
form a thin later on the earth surface leads to the hydrostatic
approximation which reduces the third momentum equation to
\eqref{eq:PEHydroStatic}. Beyond its obvious numerical significance,
this anisotropy in the governing equations has many interesting
theoretical consequences. We refer the interested reader to the
classical texts \cite{RoisinBeckers} and \cite{Pedlosky} for an
introduction from the physical point of view.

Particularly in view of the numerous complications involved in
extending the existing deterministic model to the stochastic setting
we have made some simplifications for the purposes of clarity of
presentation. The equation (\ref{eq:PE2DBasic}) is a description of
the earth's ocean but all of what follows can be easily extended to
the PEs of the atmosphere or of the coupled atmosphere-ocean system
(see \cite{LionsTemamWang3}). We assume moreover that the
$\beta$-plane approximation is valid. This assumption, that the earth
is locally flat, is appropriate for regional climatological
studies. Of course, for larger scales one must include additional
terms that account for the curvature of the earth. Since it is
convenient to work in the rotating reference frame of the earth's
surface, an additional antisymmetric term appears in the momentum
equations. The Coriolis parameter in this term, which we denote by
$f$, depends on the earth's angular velocity and the local latitude of
the region under investigation. In the context of the $\beta$-plane
approximation, $f$ is usually a linear function of $y$, $f = f_{0}(1 +
\beta y)$.
Here we take $f$ to be constant, but once again the proof
is easily modified to treat the more general case.

Several other terms have been simplified or deleted which may be
reintroduced in their full form with no new complications to the
mathematical framework or to the proof of the main theorem. We neglect
the density dependence on the salinity of the ocean. We therefore drop
the diffusion equation that accounts for variations in salt
concentration in the fluid. We also ignore further, possibly
anisotropic, diffusion terms that may appear in both the momentum and
temperature equations to account for subgrid scale processes, the so
called eddy diffusion terms. Finally, as noted above, we consider only
the case of homogenous boundary conditions.

Dating back to a series of seminal works in the early 90's
\cite{LionsTemamWang1}, \cite{LionsTemamWang2}, and
\cite{LionsTemamWang3} a significant mathematical literature has
developed around the Primitive Equation. In a significant
breakthrough, the global well posedness in 3-D was established
\cite{CaoTiti}, \cite{Kob06}, \cite{Kob07}. Subsequent work of
\cite{ZianeKukavica} developed alternative proofs, which allow for
the treatment of physically relevant boundary conditions. For the
two dimensional deterministic setting we mention
\cite{PetcuTemamWirosoetisno}, \cite{BreschKazhikhovLemoine} where
both the cases of weak and strong solutions are considered. Despite
these breakthroughs in the 3-D system, the 2-D primitive equations
seem to be significantly more difficult mathematically than the 2-D
Navier-Stokes equations. For instance, it is still an open problem
as to whether weak solutions of the Primitive equations in the
deterministic setting are unique. This is a classical exercise for
the 2-D Navier Stokes equations. In any case we refer the interested
reader to the recent survey papers \cite{RousseauTemamTribbia} and
\cite{PetcuTemamZiane} (appearing in \cite{TemamTribbia}) which
provide a systematic overview of deterministic theory. Note that, in
regards to notational conventions and earlier deterministic results
the present article relies heavily on this later work.

While the deterministic mathematical theory is now on a firm ground the
stochastic theory remain underdeveloped. In
\cite{GlattHoltzZiane1} the existence of pathwise, $z$-weak solutions
was established for a simplified model with nonlinear multiplicative
noise and non-physical boundary conditions.   A more extended system
was considered in \cite{EwaldPetcuTemam} again for the so--called $z$-weak solutions
but with additive noise and periodic
boundary conditions. Adapting the methods of \cite{CaoTiti} the 3-d
case with additive noise and nonphysical boundary conditions was
recently treated in \cite{GuoHuang}.

In contrast, beginning with the seminal work \cite{BensoussanTemam},
extensive investigations for the stochastic Navier-Stokes equations have
been undertaken. For weak or martingale solutions we mention
\cite{Viot1}, \cite{Cruzeiro1},
\cite{CapinskiGatarek}, \cite{Flandoli1}, \cite{MikuleviciusRozovskii4}
and further references therein. Regarding pathwise solutions we
mention \cite{ZabczykDaPrato2}, \cite{Breckner},
\cite{BensoussanFrehse}, \cite{BrzezniakPeszat},
\cite{MikuleviciusRozovskii2}. In recent joint work of the first
coauthor \cite{GlattHoltzZiane2} the local and global theory of
pathwise solution in $H^1 = W^{1,2}$ was established. Some of the
tools and techniques developed in this final reference play a central
role herein.

In the present work we will establish the global existence and
uniqueness of a pathwise solution to (\ref{eq:PE2DBasic}),
supplemented by \eqref{eq:PE2DBasicBC} for all $U_0 = (u,v, T) \in (H^1)^3$.
We conclude
this introduction with an outline of the basic difficulties we
encounter along with the main steps in the proof.

\subsection{Basic Estimates and some Difficulties Particular to the
  Stochastic Case}
\label{sec:BasicEst}

The first step in the proof is to establish the local existence, up to
a strictly positive stopping time $\tau$, of a solution $U$ for
(\ref{eq:PE2DBasic}) in $L^\infty_t H^1_x \cap L^2_tH^2_x$. Here and
throughout the rest of the work $U$ stands for the (prognostic)
unknowns in the problem, $U = (u,v,T) = (\mathbf{v},T)$; $U^{(n)}$ will
denote some Galerkin approximation of $U$. Having implemented a
Galerkin scheme the passage to the limit is delicate as it is not
evident a priori how to uniformly choose $\tau > 0$ such that
\begin{displaymath}
  \sup_n \E \left(  \sup_{0 \leq t \leq \tau} |U^{(n)}|^2_{H^1} + \int_0^\tau |U^{(n)}|^2_{H^2}
  \right) < \infty.
\end{displaymath}
 Even if such a $\tau$ were to be found it would remain
unclear how to infer the necessary sub-sequential (strong)
compactness without changing the underlying stochastic basis. To
overcome these difficulties we follow \cite{GlattHoltzZiane2} and
perform Cauchy type estimates for the Galerkin solutions
$\{U^{(n)}\}_{n \geq 1}$ associated with (\ref{eq:PE2DBasic}) up to
a carefully chosen sequence of stopping times. Since we have
sufficient uniform control of the growth of $U^{(n)}$ at time zero
we are able to pass to the limit almost surely up to a strictly
positive time. Note that this stage of the investigation required us
to establish some novel bounds on the nonlinear portion of the
equation in $H^1$ (see (\ref{eq:H1EstBU}) below) and to make careful
use of the equivalence of some fractional order spaces. Since a
significant portion of this analysis is non-probabilistic in
character, we have separated these delicate and technical points to
a separate work, \cite{GlattHoltzTemam1}.

With a local solution $(U,\tau) = ((u,v,T), \tau)$ in hand, further a posteriori
estimates are needed to preclude the possibility of a finite time
blowup. In previous work in the deterministic setting (which
corresponds to the admissible case, $\sigma \equiv 0$) successive
estimates on $U$, $\pd{z} u$ and $\pd{x} u$ in $L^\infty_t L^2 \cap
L^2_tH^1$ were conducted to finally obtain an estimate for $U$ in
$L^\infty_t H^1_{x} \cap L^2_t H^2_{x}$. See \cite{PetcuTemamZiane}.
For the present stochastic setting several difficulties emerge which
prevent a trivial repetition of these estimates.

The first difficulty appears when one tries to make estimates for
$\pd{z} u$. If, on the one hand, we take $\pd{z}$ of
(\ref{eq:PEMoment1}) and then apply It\={o}'s formula to determine an
evolution equation for $|\pd{z} u|^2_{L^2(\mathcal{M})}$, we encounter
terms of the form
\begin{displaymath}
  \int_{\mathcal{M}} \pd{zzz} u \pd{z} u d\mathcal{M}.
\end{displaymath}
Due to (\ref{eq:BCTop}) and (\ref{eq:BCBot}) second order terms occur
on the boundary that seem to be intractable a priori. If, on the other
hand, following \cite[Section 3.3.4]{PetcuTemamZiane}, we attempt to
multiply (\ref{eq:PEMoment1}) by $Q(-\pd{zz} u)$ it is not clear what
the appropriate stochastic interpretation of $du \cdot Q (-\pd{zz} u)$
should be. Here $Q$ is the orthogonal complement of the vertical
averaging operator and is needed to get rid of the pressure in
the governing equations (cf. \eqref{eq:QuzzevolutionEqnRandomPDE} and
the remarks immediately following).

To address these difficulties we introduce an auxiliary linear
stochastic evolution system with a diffusion governed by the now
established local solution of the original system. We use this
system to ``subtract off'' the noise terms from (\ref{eq:PEMoment1})
at the cost of a number of new random terms which we must estimate.
While we are indeed able to treat these terms, at each order our
estimates require almost sure bounds (in $\omega$) on the norms of the
solution at the previous order. For this reason an involved stopping
time argument must be employed at the final step. Here we make repeated use
of a novel abstract result concerning a generic class of stochastic
processes (see Proposition~\ref{thm:stArg1}) which streamlines the
analysis.

\section{Abstract Setting}

We begin with a review of the mathematical setting for the stochastic
Primitive Equations and define the pathwise solutions we will consider
in this work. The deterministic and stochastic preliminaries are
treated successively. For the deterministic elements we largely follow
\cite{PetcuTemamZiane}, to which we refer the reader for a more
detailed treatment. For more theoretical background on the general
theory of stochastic evolution systems we mention the classical book
\cite{ZabczykDaPrato1} or the more recent treatment in
\cite{PrevotRockner}.

\subsection{The Hydrostatic Approximation}
\label{sec:HydrostaticAprox}

The hydrostatic approximation, in concert with the incompressibility and
the boundary conditions leads one to several simple observations that
allow a useful reformulation of \eqref{eq:PE2DBasic}. This will
motivate the mathematical set-up below.

First we consider the third component of the flow $w$. Notice that by
integrating \eqref{eq:PEDivFree} and making use of the boundary
condition \eqref{eq:BCTop} for $w$ we infer that
\begin{equation}\label{eq:diagonosticVar}
  w(x,z) = - \int^0_z \pd{z} w(x, \bar{z}) d \bar{z}
            = \int^0_z \pd{x} u(x, \bar{z}) d\bar{z}.
\end{equation}
Accordingly $w = w(u)$ is seen to be an explicit functional of
$u$\footnote{Indeed, $w$, $p$ and $\rho$ are called \emph{diagnostic}
  variables in geophysical fluid mechanics. By opposition $u$, $v$ and
  $T$ are referred to as \emph{prognostic} variables and are the
  unknowns in an initial value problem which we set up below.}. Also
notice that according to the boundary conditions \eqref{eq:BCTop},
\eqref{eq:BCBot} we impose on $w$, $\smallint_{-h}^0 \pd{x} u d
\bar{z} = 0$. This implies that $\smallint_{-h}^0 u d \bar{z}$ is
constant in $x$ and so, due to the lateral boundary condition
\eqref{eq:BCSides}, we conclude that
\begin{equation}\label{eq:divFreeStandIn2D}
  \int_{-h}^0 u dz = 0.
\end{equation}

Next we consider the pressure.  By integrating the hydrostatic balance
equation \eqref{eq:PEHydroStatic} and making use of the linear dependence
of the density on the temperature \eqref{eq:PEDensityRelationEmp} we deduce
\begin{equation}\label{eq:pressureDecomp}
  p_s(x) - p(x,z) = \int_z^0 \pd{z} p(x,\bar{z}) d \bar{z}
                  = - g \rho_0 \int_z^0 ( 1 - \beta_T ( T(x,\bar{z}) - T_0)) d \bar{z}.
\end{equation}
Here $p_s$ is the surface pressure, which is unknown and a function of the
horizontal variable only.  We have therefore decomposed the pressure
into two components, the second of which couples the first momentum
equation to the heat diffusion equation.  Rearranging above and taking a
partial derivative in $x$ we arrive at
\begin{equation}\label{eq:pressureDecomp}
  \pd{x}p =  \pd{x}p_s - \beta_T g \rho_0 \int_z^0  \pd{x} T d \bar{z}.
\end{equation}

With the above considerations we now rewrite \eqref{eq:PE2DBasic} as:
\begin{subequations}\label{eq:PE2Dreform}
  \begin{gather}
    \begin{split}
    \pd{t} u + u \pd{x} u + w(u) \pd{z} u - \nu \Delta u - f v +
    \pd{x} p_s
    &- \beta_T g \rho_0 \int_z^0  \pd{x} T d \bar{z}  \\
    =& F_u
      +\sigma_u(\mathbf{v},T) \dot{W}_1,
   \end{split}
   \label{eq:PEMoment1R}\\
    \pd{t} v + u \pd{x} v + w(u) \pd{z} v - \nu \Delta v + f u = F_v + \sigma_v(\mathbf{v},T) \dot{W}_2,
    \label{eq:PEMoment2R}\\
   w(u) = \int^0_z \pd{x} u d \bar{z}, \quad   \int_{-h}^0 u dz = 0,
    \label{eq:PEDivFreeR}\\
    \pd{t} T + u \pd{x} T + w(u) \pd{z} T - \mu \Delta T = F_T + \sigma_T(\mathbf{v},T) \dot{W}_3,
    \label{eq:PETempCoupleR}
  \end{gather}
\end{subequations}

\subsection{Basic Function Spaces}
\label{sec:BasicFnSpaces}

The main function spaces used are defined as follows.  Take:
\begin{displaymath}
  \begin{split}
    H := \left\{U = (u, v,T) \in L^2(\mathcal{M})^3: \int_{-h}^0 u dz = 0 \right\}.
 \end{split}
\end{displaymath}
We equip $H$ with the inner product\footnote{One sometimes also finds the more general definition
  $(U, U^{\sharp}) := \int_{\mathcal{M}} \mathbf{v}\cdot
  \mathbf{v}^\sharp d \mathcal{M} + \kappa \int_{\mathcal{M}} T T^\sharp d
  \mathcal{M}$ with $\kappa > 0$ fixed. This $\kappa$ is useful for
  the coherence of physical dimensions and for (mathematical)
  coercivity. Since this is not needed here we take $\kappa =1$.}
\begin{displaymath}
  (U,U^\sharp) := \int_{\mathcal{M}} \mathbf{v}\cdot \mathbf{v}^\sharp d \mathcal{M} + \int_{\mathcal{M}}  T T^\sharp d \mathcal{M}, \quad
  U = (\mathbf{v}, T), U^\sharp = (\mathbf{v}^\sharp, T^\sharp).
\end{displaymath}
Here and below we shall make use of the vertical averaging operator
$\mathit{P}\phi = \frac{1}{h} \smallint_{-h}^0 \phi(\bar{z}) d\bar{z}$
and its orthogonal complement $\mathit{Q}\phi = \phi - \mathit{P}
\phi$. Note that the projection operator $\Pi: L^2(\mathcal{M})^3
\rightarrow H$ may be explicitly defined according to $U \mapsto
(\mathit{Q} u, v, T)$.  We also define
\begin{displaymath}
  \begin{split}
    V := \left\{ U = (u, v, T) \in H^1(\mathcal{M})^3:
    \int_{-h}^0 u dz = 0,
    \mathbf{v} = 0 \textrm{ on } \Gamma_l \cup \Gamma_b \right\}.
  \end{split}
\end{displaymath}
Here we take the inner product $((\cdot, \cdot)) = \nu ((\cdot, \cdot))_1
+ \mu (( \cdot, \cdot ))_2$ where, for given $U = (\mathbf{v},T), U^\sharp =(\mathbf{v}^\sharp, T^\sharp)$
\begin{displaymath}
  \begin{split}
    ((U,U^\sharp))_1 &:=
    \int_{\mathcal{M}} \pd{x} \mathbf{v} \cdot \pd{x} \mathbf{v}^\sharp + \pd{z} \mathbf{v} \cdot \pd{z} \mathbf{v}^\sharp \,d \mathcal{M} +
    \alpha_{\mathbf{v}}\int_{\Gamma_i}  \mathbf{v} \cdot \mathbf{v}^\sharp \,dx,\\
    ((U,U^\sharp))_2 &:=
    \int_{\mathcal{M}} \pd{x}T\pd{x} T^\sharp + \pd{z}T \pd{z} T^\sharp\,d \mathcal{M} +
    \alpha_T \int_{\Gamma_i} T  T^\sharp \,dx. \\
  \end{split}
\end{displaymath}
Note that under these definitions a Poincar\'{e} type
inequality  $|U|
\leq C \|U\|$ holds for all $U \in H^1(\mathcal{M})^3 \supset V$.
Moreover the norms $\| \cdot \|_{H^{1}}$, $\| \cdot \|$ may be seen to
be equivalent over all of $H^1(\mathcal{M})^3$.

Even if $U$ is very regular many of the main terms in the abstract
formulation of (\ref{eq:PE2Dreform}) do not belong to $V$ (see
(\ref{eq:linLowerOrderDef}),(\ref{eq:NLTerm1}), (\ref{eq:NLTerm})) As
such, we shall also make use of some additional auxiliary spaces:
\begin{displaymath}
  \begin{split}
    \tilde{V} &:= \left\{ U = (u, v, T) \in H^1(\mathcal{M})^3:
        \int_{-h}^0 u dz = 0,
      \mathbf{v} = 0 \textrm{ on } \Gamma_l \right\},\\
    \mathcal{Z} &:= \left\{ U = (u, v, T) \in H^1(\mathcal{M})^3:
      \mathbf{v} = 0 \textrm{ on } \Gamma_l  \right\}.
  \end{split}
\end{displaymath}
As for $V$ we endow both spaces with the norm $\| \cdot \|$.  One
may verify that $\Pi:
\mathcal{Z} \rightarrow \tilde{V}$ and is continuous on $H^1(\mathcal{M})^3$.


Finally we take $V_{(2)} = H^2(\mathcal{M})^3 \cap V$ and
equip this space with
the classical $H^2(\mathcal{M})$ norm which we denote by $| \cdot |_{(2)}$.
Since a considerable portion of the work below will consist in making
estimates for the first momentum equation (\ref{eq:PE2DBasic}) (or
equivalently (\ref{eq:PEMoment1R})) we set for simplicity
\begin{displaymath}
  | u |_{L^2(\mathcal{M})} := |u|, \quad
  |\nabla u |_{L^{2}(\mathcal{M})} := \|u\|, \quad
  | u |_{H^2(\mathcal{M})} := | u |_{(2)}, \quad
\end{displaymath}
for $u \in L^{2}(\mathcal{M})$ or $H^{1}(\mathcal{M})$ or $H^{2}(\mathcal{M})$.  Note that since
we will always use a lower case $u$ (or as needed $u^\sharp$,
$u^\flat$) for the first component of elements in the spaces $H, V,
V_{(2)}$ the context will be clear.

\subsection{The deterministic framework}
\label{sec:determ-fram}

The linear second order terms in the equation are captured in the Stokes-type
operator $A$ which is understood as a bounded operator from
$V$ to $V'$ via $\langle A U, U^{\sharp} \rangle = ((U, U^{\sharp}))$. The additional terms in the
variational formulation of this portion of the equation
capture the Robin boundary condition (\ref{eq:BCTop}). They may
be formally derived by multiplying $-\nu\Delta u,  -\nu \Delta v, -\mu \Delta T$ in
\eqref{eq:PEMoment1R}, \eqref{eq:PEMoment2R}, \eqref{eq:PETempCoupleR} by test functions
$u^{\sharp}, v^{\sharp}, T^{\sharp}$, integrating over $\mathcal{M}$ and integrating
by parts.   We shall make use of the subspace $D(A) \subset V_{(2)}$ given by
\begin{displaymath}
  \begin{split}
    D(A) = \{ U = (\mathbf{v}, T) \in V_{(2)}:\ &
       \pd{z} \mathbf{v} +\alpha_{\mathbf{v}} \mathbf{v} = 0,
       \pd{z} T + \alpha_T T = 0 \textrm{ on } \Gamma_i,\\
    &\pd{x} T = 0 \textrm{ on } \Gamma_l,
       \pd{z} T = 0 \textrm{ on } \Gamma_b
       \}.
 \end{split}
\end{displaymath}
On this space we may extend $A$ to an unbounded operator by defining
\begin{displaymath}
  AU = \left(
  \begin{split}
    -\nu \mathit{Q} \Delta u\\
    -\nu \Delta v\\
    -\mu \Delta T\\
  \end{split}
  \right), \quad U \in D(A).
\end{displaymath}

Since $A$ is self adjoint, with a compact inverse $A^{-1} : H
\rightarrow D(A)$ we may apply the standard theory of compact,
symmetric operators to guarantee the existence of an orthonormal basis
$\{\Phi_k\}_{k \geq 0}$ for $H$ of eigenfunctions of $A$ with the
associated eigenvalues $\{\lambda_k\}_{k \geq 0}$ forming an unbounded,
increasing sequence. Note that by the regularity results in
\cite{Ziane1} or \cite{TemamZiane1} we have $\Phi_k \in D(A) \subset
V_{(2)}$. Define
\begin{displaymath}
  H_n = span \{\Phi_1, \ldots, \Phi_n\}.
\end{displaymath}
Take $P_n$ and $Q_n = I - P_n$ to be the projections from $H$ onto
$H_n$ and its orthogonal complement respectively.   For $m > n$ let
$P^{n}_{m} = P_{m} -P_{n}$.

Note that in some previous works, the second component of
the pressure (cf. \eqref{eq:pressureDecomp} and
\cite[Section 2]{PetcuTemamZiane}), is included in the
definition of the principal linear operator $A$.  Since this
breaks the symmetry of $A$ we relegate such terms to a
separate, lower order operator $A_p$,  which we define from $V'$ via
$\langle A_p U, U^\sharp \rangle:=  \kappa g \rho_0
\int_{\mathcal{M}} \int_z^0  T d \bar{z} \pd{x} u^\sharp
d\mathcal{M}, \forall U^\sharp\in V.$  Taking into account the boundary conditions for
$u^\sharp$ on $\Gamma_\ell\enspace (x=0,L),$ this may be extended to a map
$A_{p}: V \rightarrow H$ via
\begin{equation}\label{eq:linLowerOrderDef}
  A_p U =   \left(
  \begin{array}{c}
    -\beta_T g \rho_0 Q \left( \int_z^0 \pd{x} T d \bar{z} \right) \\
    0\\
    0
  \end{array}
  \right).
\end{equation}
If $U \in D(A)$, $A_{p} U \in \tilde{V}$ and we have that
\begin{equation}\label{eq:estimateAloworder}
  \begin{split}
    | A_p U|
      \leq c \|U\|, \quad
    \| A_p U \|
      \leq c  |U|_{(2)}.\\
   \end{split}
\end{equation}

We next capture the nonlinear portion of (\ref{eq:PE2DBasic}).
Accordingly we \emph{define} the diagnostic function $w$ by setting
\begin{equation}\label{2.7a} 
  w(U) = w(u) = \int^0_z \pd{x} u d \bar{z}, \quad
  U = (u, v, T) \in V.
\end{equation}
For $U =(\mathbf{v}, T), U^\sharp=(\mathbf{v}^\sharp,T^\sharp) \in V$ we take $B(U,U^\sharp) = B_1(U,U^\sharp) + B_2(U,U^\sharp)$ where
\begin{equation}\label{eq:NLTerm1}
  B_1(U,U^\sharp) :=  \left(
  \begin{split}
    \mathit{Q} (u \pd{x}u^\sharp)\\
    u \pd{x}v^\sharp\\
    u \pd{x}T^\sharp\\
  \end{split}
  \right)
    = \left(
  \begin{split}
    B_1^1(u,u^\sharp)\\
    B_1^2(u,v^\sharp)\\
    B_1^3(u,T^\sharp)\\
  \end{split}
  \right)
\end{equation}
and
\begin{equation}\label{eq:NLTerm}
  B_2(U,U^\sharp) :=  \left(
  \begin{split}
    \mathit{Q} (w(u) \pd{z} u^\sharp)\\
    w(u) \pd{z} v^\sharp \\
    w(u) \pd{z} T^\sharp \\
  \end{split}
  \right)
  = \left(
  \begin{split}
    B_2^1(u,u^\sharp)\\
    B_2^2(u,v^\sharp)\\
    B_2^3(u,T^\sharp)\\
  \end{split}
  \right).
\end{equation}
We also set $B^j = B^j_1 + B^j_2$, $j = 1,2,3$.  We summarize some
properties of $B$ needed in the sequel
\begin{Lem}\label{thm:Best}
  $B$ is well defined as a bilinear and continuous map
  from $V\times
  V$ to $V',$ from $V \times V_{(2)}$ and $V_{(2)}\times V$ to $H$.
  Moreover $B$ satisfies the following properties and estimates:
  \begin{itemize}
  \item[(i)] For any $U, U^\sharp \in V$ and $\langle
    B(U,U^\sharp),U^\sharp  \rangle = 0$.
  \item[(ii)] For $U, U^\sharp, U^\flat \in V$
    \begin{equation}\label{eq:3dtypeEstimateFor2DPE}
      | \langle B(U, U^\sharp), U^\flat \rangle |
      \leq c \|U \| \|U^\sharp \| | U^\flat |^{1/2} \| U^\flat \|^{1/2}.
   \end{equation}
  \item[(iii)] On the other hand if we assume that $U \in V$ $U^\sharp \in
    V_{(2)}$
    and $U^\flat \in H$ then
    \begin{equation}\label{eq:strongTypeEstimate}
      |\langle B(U, U^\sharp), U^\flat \rangle|
        \leq c \|U\| \|U^\sharp\|^{1/2} |U^\sharp|^{1/2}_{(2)} |U^\flat|.
    \end{equation}
    In particular, for $U \in V_{(2)}$,
    \begin{equation}\label{eq:L2NormBUv2}
      |B(U,U)|^2 \leq c\|U\|^{3} |U|_{(2)}.
    \end{equation}

    Also if $U = (\mathbf{v}, T) = (u,v, T) \in V_{(2)}$,
$U^\sharp\in V$ and $U^b\in H,$ then
    \begin{equation}\label{eq:strongTypeEstimateFirstCompControl}
      | \langle B(U, U^\sharp), U^\flat \rangle | \leq c \|u\|^{1/2} |u|_{(2)}^{1/2}\|U^\sharp\| |U^\flat|.
    \end{equation}
  \item[(iv)] For $U \in V_{(2)}$, $B(U) \in \tilde{V}$ and satisfies the estimate

   \begin{equation}\label{eq:H1EstBU}
     \begin{split}
      \|B(U,U)\|^2
       \leq& c \|U\| |U|^3_{(2)}.
   \end{split}
   \end{equation}
  \item[(v)] Given $U, U^\sharp \in V_{(2)}$,  $U^\flat \in H$
    \begin{equation}\label{eq:genericFirstCompEstL2ClassComp1}
      \begin{split}
        | \langle B^1_1(u, u^\sharp), u^\flat \rangle|
          \leq c
          |u |^{1/2} | u |_{(2)}^{1/2}
          | \pd{x} u^\sharp | |u^\flat |,
        \end{split}
    \end{equation}
    \begin{equation}\label{eq:genericFirstCompEstL2ClassComp2}
      \begin{split}
        | \langle B^1_1(u, u^\sharp), u^\flat \rangle|
        \leq c
        |u |^{1/2} \|u \|^{1/2}
        |\pd{x}u^\sharp |^{1/2}  \|\pd{x}u^\sharp \|^{1/2} |u^\flat |.\\
      \end{split}
    \end{equation}
    On the other hand
    \begin{equation}\label{eq:genericFirstCompEstL2FuckedComp1}
      \begin{split}
        | \langle B^1_2(u, u^\sharp), u^\flat \rangle|
        \leq c
        |\pd{x} u |
        |\pd{z} u^\sharp |^{1/2}
        \|\pd{z} u^\sharp \|^{1/2}
        |u^\flat |,
      \end{split}
    \end{equation}
    \begin{equation}\label{eq:genericFirstCompEstL2FuckedComp2}
      \begin{split}
        | \langle B^1_2(u, u^\sharp), u^\flat \rangle|
        &\leq c
        \| u \|^{1/2}  |u |_{(2)}^{1/2}
        | \pd{z} u^\sharp |
        |u^\flat |.
      \end{split}
    \end{equation}
    \item[(vi)] For $U = (\mathbf{v}, T) \in D(A)$
    \begin{equation}\label{eq:semiCancelpdz}
      \langle B^1(u, u), -\pd{zz} u\rangle =
       - \frac{2}{h} \int_\mathcal{M} u \pd{x} u ( \alpha_{\mathbf{v}} u(x,0) + \pd{z} u(x, -h ) d\mathcal{M}
    \end{equation}
    which admits the estimate
    \begin{equation}\label{eq:semiCancelpdzEst}
      \begin{split}
      |\langle B^1(u, u), -\pd{zz} u\rangle|
         &\leq c (|u|\| u \|^2
             +|\pd{z}u|^{1/2} \|\pd{z}u\|^{1/2}
               |u|^{1/2} \| u \|^{3/2}).
     \end{split}
   \end{equation}
\end{itemize}
\end{Lem}

The continuity properties of $B$ as well as the
basic cancellation property (i) are well established in the literature.
The estimates \eqref{eq:genericFirstCompEstL2ClassComp1},
\eqref{eq:genericFirstCompEstL2ClassComp2} may be
established as for the classical Navier-Stokes systems (see, for
example, \cite{Temam1}). On the other hand the estimates
\eqref{eq:3dtypeEstimateFor2DPE}, \eqref{eq:strongTypeEstimate},
\eqref{eq:strongTypeEstimateFirstCompControl},
\eqref{eq:genericFirstCompEstL2FuckedComp1},
\eqref{eq:genericFirstCompEstL2FuckedComp2},
\eqref{eq:semiCancelpdzEst}
may be proved with
anisotropic techniques. See \cite{TemamZiane1} or
\cite{GlattHoltzZiane1}. The property \eqref{eq:H1EstBU},
which is new and requires extensive computations,
may be found in \cite{GlattHoltzTemam1}.

We next capture the Coriolis forcing with the bounded operator
$E: H \rightarrow H$ given by
\begin{equation}\label{eq:CorTerm}
  EU  :=  \left(
  \begin{array}{c}
    -Q f v\\
    f u\\
    0
  \end{array}
  \right).
\end{equation}
We observe that $E$ is also continuous from $V$ to $\tilde{V}$ and
that
\begin{equation}  \label{eq:CorbndOperator}
 | EU | \leq c|U|, \quad \|EU\| \leq c\|U\|.
\end{equation}
Finally, for brevity of notation we shall sometimes write
\begin{equation}\label{eq:nonlinearExt}
  N(U) = A_{p}U + B(U, U) + EU, \quad U \in V.
\end{equation}

\subsection{The stochastic framework: nonlinear, multiplicative white noise forcing}
\label{sec:stoch-fram-non}

It finally remains to define the white noise driven terms in
\eqref{eq:PE2DBasic}.  To begin we fix a stochastic basis
$\mathcal{S}:= (\Omega, \mathcal{F}, \{\mathcal{F}_t\}_{t \geq 0}, \mathbb{P},
\{W^k\}_{k \geq 1})$, that is a filtered probability space
with $\{W^k\}_{k \geq 1}$ a sequence of
independent standard 1-D Brownian motions relative to the filtration $\mathcal{F}_t$.  In
order to avoid unnecessary complications below we may assume that
$\mathcal{F}_t$ is complete and right continuous (see \cite{ZabczykDaPrato1}).
Fix a separable Hilbert space $\mathfrak{U}$ with an associated
orthonormal basis $\{e_k\}$. We
may formally define $W$ by taking $W = \sum_{k} W^k
e_k$.
As such $W$ is a cylindrical Brownian motion evolving over
$\mathfrak{U}$.

We next recall some basic definitions and properties of spaces of
Hilbert-Schmidt operators.  For this purpose we suppose that $X$ and
$\tilde{X}$ are any separable Hilbert spaces with the associated norms and  inner
products given by $| \cdot |_X$, $| \cdot |_{\tilde{X}}$ and $\langle \cdot, \cdot \rangle_{X}$ $\langle \cdot,
\cdot \rangle_{\tilde{X}}$, respectively.  We denote by
\begin{displaymath}
 L_2(\mathfrak{U}, X) = \{ R \in \mathcal{L}(\mathfrak{U},X): \sum_k |Re_k|^2_{X} < \infty \},
\end{displaymath}
the collection of Hilbert Schmidt operators from $\mathfrak{U}$ to
$X$. By endowing this collection with the inner product
\begin{displaymath}
  \langle R, S \rangle_{L_2(\mathfrak{U}, X)} = \sum_k \langle R e_k, S e_k \rangle_X,
\end{displaymath}
we may consider $L_2(\mathfrak{U},X)$ as itself being a Hilbert space.
One may readily show that if $R^{(1)} \in L_2(\mathfrak{U}, X)$ and
$R^{(2)} \in L(X,\tilde{X})$ then indeed $R^{(2)}R^{(1)} \in
L_2(\mathfrak{U}, \tilde{X})$.

Given an $X$-valued predictable\footnote{For a given  stochastic basis
  $\mathcal{S}$, let $\Phi = [0,\infty)\times\Omega$ and take
  $\mathcal{G}$ to be the $\sigma$-algebra generated by sets of the
  form
  \begin{displaymath}
    (s,t] \times F, \quad 0 \leq s< t< \infty, F \in \mathcal{F}_s;
    \quad \quad
    \{0\} \times F, \quad F \in \mathcal{F}_0.
  \end{displaymath}
  Recall that a $X$ valued process $U$ is called predictable (with
  respect to the stochastic basis $\mathcal{S}$) if it is measurable
  from $(\Phi,\mathcal{G})$ into $(X, \mathcal{B}(X))$,
  $\mathcal{B}(X)$ being the family of Borel sets of $X$.} process $G \in
L^{2}(\Omega; L^{2}_{loc}([0, \infty),L_{2}(\mathfrak{U}, X)))$ one
may define the (It\={o}) stochastic integral
\begin{displaymath}
   M_{t} := \int_{0}^{t} G dW = \sum_k \int_0^t G_k dW^k,
\end{displaymath}
as a square integrable function from $\Omega$ into $X.$  Furthermore
$M_t$ is an element of $\mathcal{M}^2_X$, that is the space of all
$X$-valued square integrable martingales (see \cite[Section
2.2, 2.3]{PrevotRockner}), and, as such, $\{M_t \}_{t \geq 0}$ has
many desirable properties.  Most notably the Burkholder-Davis-Gundy (BDG)
inequality holds which in our context takes the form
\begin{equation}\label{eq:BDG}
  \E \left(\sup_{ t' \in [0,t]} \left| \int_0^{t'} G dW  \right|_X \right)
  \leq c\enspace \E \left(
    \int_0^{t} |G|_{L_2(\mathfrak{U}, X)}^2 \right)^{1/2},
\end{equation}
for any $t > 0$, where $c$ is here an absolute constant.

Given any Banach spaces $\mathcal{X}$ and $\mathcal{Y}$ we denote by
$Bnd_u(\mathcal{X}, \mathcal{Y})$, the collection of all mappings
\begin{displaymath}
  \Psi: \Omega \times [0, \infty) \times \mathcal{X} \rightarrow \mathcal{Y} ,
\end{displaymath}
such that $\Psi$ is almost surely continuous in $[0,\infty) \times \mathcal{X}$ and
\begin{displaymath}
    \| \Psi(x) \|_{\mathcal{Y}} \leq c(1 + \|x\|_{\mathcal{X}}), \quad x
    \in \mathcal{X},\\
\end{displaymath}
where the numerical constant $c$ may be chosen
independently of $t$ and $\omega$.  If in addition
\begin{displaymath}
   \| \Psi(x) - \Psi(y)\|_{\mathcal{Y}} \leq c \|x - y \|_{\mathcal{X}},
   \quad x, y \in \mathcal{X}\\
\end{displaymath}
we say that $\Psi$ is in $Lip_u(\mathcal{X}, \mathcal{Y})$.

With these notations now in place we define
\begin{equation}\label{eq:Hdef}
  \sigma(U) =
  \left(
  \begin{array}{c}
    Q \sigma_u(\mathbf{v},T)\\
    \sigma_v(\mathbf{v},T)\\
    \sigma_T(\mathbf{v},T)
  \end{array}
  \right)
\end{equation}
We shall assume throughout this work that
\begin{displaymath}
  \sigma: \Omega \times [0, \infty) \times H \rightarrow L_2(\mathfrak{U},H)
\end{displaymath}
such that
\begin{equation}\label{eq:measurablityConditions}
  \begin{split}
    \textrm{If }U \textrm{ is an } &H\textrm{-valued, predictable process, then}\\
    & \sigma(U)  \textrm{ is an $L_2(\mathfrak{U},H)$-valued,
   predictable process,}
  \end{split}
\end{equation}
and
\begin{equation}\label{eq:lipCond}
  \sigma \in Lip_u(H, L_2(\mathfrak{U}, H))
   \cap Lip_u(V, L_2(\mathfrak{U}, V))
   \cap Bnd_u(V, L_2(\mathfrak{U}, D(A))).
\end{equation}
Note that under the conditions imposed above the stochastic integral
$\int_0^\tau \sigma(U) dW$ may be shown to be well defined, taking
values in $H$ for any $H$ predictable $U \in L^2(\Omega,
L^2_{loc}([0,\infty); H))$.  Denoting $\sigma_k(\cdot) = \sigma(\cdot)
e_k$ we may
interpret this integral in the expansion\footnote{To recover the
  formulation of the stochastic forcings in \eqref{eq:PE2DBasic},  \eqref{eq:noiseFormalExp} we may consider the special
  case where
  \begin{displaymath}
    \begin{split}
      \sigma^k_u \equiv 0 &\textrm{ when } k = 0 \, (mod \, 3) \\
      \sigma^k_v \equiv 0 &\textrm{ when } k = 1 \, (mod \, 3) \\
      \sigma^k_T \equiv 0 &\textrm{ when } k = 2 \, (mod \, 3) \\
    \end{split}
  \end{displaymath}
  and take $\dot{W}_1 = \sum_k \dot{W}^{3k} e_{3k}$, $\dot{W}_2 = \sum_k \dot{W}^{3k+1} e_{3k
    +1}$, $\dot{W}_3 = \sum_k \dot{W}^{3k+1} e_{3k +2}$.
}
\begin{displaymath}
  \begin{split}
  \int_0^t \sigma(U) dW = \sum_{k \geq 1} \int_0^t \sigma^k(U) dW^k
                   = \sum_{k \geq 1} \left(
                   \begin{split}
                     \int_0^t Q\sigma_{u}^k(U) dW^k,\\
                     \int_0^t \sigma_{v}^k(U) dW^k,\\
                     \int_0^t \sigma_{T}^k(U) dW^k\\
                   \end{split}
                 \right).\\
  \end{split}
\end{displaymath}

\begin{Rmk}\label{rmk:NoiseCond}
  The condition (\ref{eq:lipCond}) may be weakened to
  \begin{equation}\label{eq:lipCondClass}
    \sigma \in Lip_u(H, L_2(\mathfrak{U}, H))
    \cap Lip_u(V, L_2(\mathfrak{U}, V))
    \cap Bnd_u(D(A), L_2(\mathfrak{U}, D(A)))
  \end{equation}
  in the proof of local and maximal existence of solutions below (see
  Proposition~\ref{thm:MaxExist}).  However, for the proof of global
  existence of solutions we need the stronger condition (\ref{eq:lipCond}).  See
  Remark~\ref{rmk:BndMods} below, for further details.  Even with this more
  restrictive condition \eqref{eq:lipCond} the theory covers a physically interesting
  class of additive and nonlinear multiplicative stochastic forcing regimes relevant 
  to the 'parametrization' problem
  discussed in the Introduction.   We refer the interested reader to
  \cite{GlattHoltzTemamTribbia1} for further details and examples.
\end{Rmk}

For the external forcing terms $F_u,F_v, F_T$ we let:
\begin{displaymath}
  F = \left(
  \begin{array}{c}
    Q F_u\\
    F_v\\
    F_T
  \end{array}
  \right).
\end{displaymath}
We assume throughout the analysis below that $F$ is an $H$-valued,
predictable process with
\begin{equation}\label{eq:SizeConditiononF}
  F \in L^2( \Omega; L^2_{loc} ([0,\infty), H)).
\end{equation}
We shall allow for the case of probabilistic dependence in the initial
data $U_{0} = (u_{0}, v_{0},T)$ as well.  Specifically we assume that
\begin{equation}\label{eq:dataCond}
  U_{0} \in L^{2}(\Omega; V) \textrm{ and is } \mathcal{F}_{0} \textrm{-measurable}.
\end{equation}

\subsection{Definition of solutions}
\label{sec:definition-solutions}

With the abstract mathematical definitions for each term in the
original system now in hand we may reformulate (\ref{eq:PE2Dreform})
as an abstract evolution equation
\begin{equation}\label{eq:mainSystemAbsDiffForm}
  \begin{split}
  d U  + (AU + N(U)) dt &= F dt + \sigma(U) dW,\\
           U(0) &= U_0.
 \end{split}
\end{equation}
More precisely we have the following basic notion of local and global
pathwise solutions to the above system.
\begin{Def}[Pathwise Strong Solutions of the Primitive Equations]
  \label{def:solutionNot}
  Let $\mathcal{S} = (\Omega, \mathcal{F}, \{\mathcal{F}_t\}_{t \geq
    0}, \mathbb{P}, W)$ be a fixed stochastic basis. Assume that $F$ is
  as in (\ref{eq:SizeConditiononF}), that $U_{0}$ satisfies \eqref{eq:dataCond} and that
  $\sigma$ satisfies (\ref{eq:measurablityConditions}), (\ref{eq:lipCond}).
  \begin{itemize}
  \item[(i)] A pair $(U, \tau)$ is \emph{a local strong (pathwise) solution} of
    (\ref{eq:mainSystemAbsDiffForm}) if $\tau$ is a strictly positive stopping
    time and $U(\cdot \wedge \tau)$ is a $\mathcal{F}_{t}$ adapted process in
    $H$ so that
    \begin{equation}\label{eq:Uregularity}
      \begin{split}
        U(\cdot \wedge \tau) \in  L^2(\Omega; C ([0,\infty); V)),\\
        U(\tau) \indFn{t \leq \tau} \in  L^2(\Omega; L^2_{loc}([0,\infty); D(A))),
      \end{split}
    \end{equation}
    and satisfies, for every $t \geq 0$ and every $\tilde{U} \in H$,
    \begin{equation}\label{eq:spdeAbstrac}
      \begin{split}
        \langle U(t \wedge &\tau), \tilde{U} \rangle
           + \int_0^{t \wedge \tau} \langle A U + N(U), \tilde{U} \rangle ds\\
           &= \langle U_{0}, \tilde{U} \rangle
                 + \int_0^{t \wedge \tau} \langle F, \tilde{U} \rangle ds
                                     + \int_0^{t \wedge \tau} \langle \sigma(U), \tilde{U} \rangle dW.
     \end{split}
   \end{equation}
 \item[(ii)] Strong solutions of \eqref{eq:mainSystemAbsDiffForm} are said to be
   \emph{(pathwise) unique} up to a stopping time $\tau > 0$ if given
   any pair of strong solutions $(U^1,\tau)$, $(U^2,\tau)$ which
   coincide at $t = 0$ on $\tilde{\Omega} = \{U^1(0) = U^2(0)\}$, then
   \begin{displaymath}
    \Prb \left( \indFn{\tilde{\Omega}} ( U^1(t \wedge \tau) - U^2(t \wedge
      \tau))  = 0; \forall t \geq 0 \right) = 1.
   \end{displaymath}
 \item[(iii)] Suppose that $\{\tau_n\}_{n\geq 1}$ is a strictly increasing sequence
    of stopping times converging to a (possibly infinite) stopping time
    $\xi$ and assume that $U$ is a continuous $\mathcal{F}_{t}$-adapted process in $H$.  We
   say that the triple $(U,\xi, \{\tau_n\}_{n\geq 1} )$ is \emph{a maximal
     strong solution} if $(U, \tau_n)$ is a local strong solution for
   each $n$ and
  \begin{equation}\label{eq:FiniteTimeBlowUp}
      \sup_{t \in [0, \xi]} \|U\|^2 + \int_0^{\xi} |A U|^2 ds = \infty
  \end{equation}
  almost surely on the set $\{\xi < \infty\}$.
  \item[(iv)] If $(U, \xi, \{\tau_n\}_{n\geq 1} )$ is a maximal strong solution and $\xi =
    \infty$ a.s. then
    we say that the solution is global.
  \end{itemize}
\end{Def}

We now have a complete mathematical framework and may state, in
precise terms, the main theorem in this work:
\begin{Thm}\label{thm:MainReslt}
  Suppose that the conditions imposed in Definition~\ref{def:solutionNot} hold.
  Then there exists a unique global solution $U$ of (\ref{eq:mainSystemAbsDiffForm}).
\end{Thm}

\section{Local and Maximal Existence and Uniqueness}

The proof of local and maximal existence of solutions for
(\ref{eq:mainSystemAbsDiffForm}) makes use of techniques developed
for the 3D Navier-Stokes Equations
\cite{GlattHoltzZiane2}.   Since the analysis here is very similar on many points
to \cite{GlattHoltzZiane2}
our treatment will be brief in some details.
However, one crucial step, to show that the
Galerkin approximations associated to \eqref{eq:mainSystemAbsDiffForm}
are Cauchy (in appropriate spaces) is quite delicate.  This is
due to stray terms that arise from the
discretization which must be controlled.  See
Proposition~\ref{thm:CompEst} below.

\begin{Prop}\label{thm:MaxExist}
  Suppose that $U_0$, $F$ satisfy the conditions imposed in
  Definition~\ref{sec:HydrostaticAprox}.  For $\sigma$ we assume
  (\ref{eq:measurablityConditions}) and may weaken
  (\ref{eq:lipCond}) to (\ref{eq:lipCondClass}).
  Then there exists a unique maximal strong solution $(U,\xi)$ for (\ref{eq:mainSystemAbsDiffForm}).
  Moreover, for any (deterministic) $t > 0$,
  \begin{equation}\label{eq:weakBnds}
    \mathbb{E} \left(\sup_{0 \leq t' \leq \xi \wedge t} |U|^2 + \int_0^{\xi
        \wedge t} \|U\|^2  dt' \right)   < \infty.
  \end{equation}
\end{Prop}

\begin{proof}
  The first step in the proof, to establish certain Cauchy
  estimates for the Galerkin approximations \eqref{eq:galerkinSystem} of \eqref{eq:mainSystemAbsDiffForm}
  is carried out in
  in Lemma~\ref{thm:CompEst}. For the details of the
  passage to limit we refer the reader to
  \cite[Proposition 4.2]{GlattHoltzZiane2} and the remarks thereafter.

  To establish local, pathwise, uniqueness in the sense of
  Definition~\ref{def:solutionNot} we note that the estimate
  \eqref{eq:3dtypeEstimateFor2DPE} of $B$ (in dimension 2) is the same as
  may be achieved for the Navier-Stokes non-linearity in $d =3$.    The
  proof is therefore identical to \cite[Proposition 4.1]{GlattHoltzZiane2}.

  With a local strong solution in hand it remains to extend this
  solution to a maximal existence time $\xi$ as in
  Definition~\ref{def:solutionNot}, (iii).  For this point we may
  employ an argument going back to \cite{Jacod1}. For a more recent
  treatment see \cite[Lemma 4.1, 4.2, Theorem 4.1]{GlattHoltzZiane2}.
  Since we have the cancellation property in $B$
  (Lemma~\ref{thm:CompEst},(i)) the bound on the weak norms up to a possible
  finite time blow up (\ref{eq:weakBnds}) may be established exactly
  as in \cite[Lemma 4.2]{GlattHoltzZiane2}
\end{proof}

\subsection{Local Cauchy Estimates for the Galerkin System}
\label{sec:CompEst}

We turn now to the task of estimating the difference of solutions of
the Galerkin system associated to \eqref{eq:mainSystemAbsDiffForm} at different
orders. We begin by recalling some definitions.  A $\mathcal{F}_{t}$-adapted
process $U^{(n)}  \in L^2(\Omega,C([0,\infty); H_n))$
is a solution of the Galerkin system of order $n$ for \eqref{eq:mainSystemAbsDiffForm} if
it satisfies:
\begin{equation}\label{eq:galerkinSystem}
  \begin{split}
    d U^{(n)} + (A U^{(n)} + P_n N(U^{(n)}))dt
        &= P_{n}F dt + P_{n} \sigma (U^{(n)}) dW,\\
        U^{(n)}(0) &= P_nU_0.
  \end{split}
\end{equation}
Note that by the standard theory of stochastic ordinary
differential equations one may establish the global existence of a unique
solution $U^{(n)}$ at each order.  See e.g. \cite{Flandoli1} for details.

\begin{Prop}\label{thm:CompEst}
  Let $\{U^{(n)}\}_{n \geq 1}$ be the (global) solutions of the Galerkin systems
  \eqref{eq:galerkinSystem} and suppose that there exists a
  deterministic constant $M$ such that
  \begin{equation}\label{eq:uniformDataBnd}
    \|U_{0}\|^2 \leq M \quad a.s.
  \end{equation}
  Then
  \begin{itemize}
  \item[(i)] there exists a stopping time $\tau$, with $\tau > 0$, a
    subsequence $n_j$ and a process $U$ almost surely in
    $C([0,\infty); V) \cap L^{2}_{loc}([0,\infty); D(A))$
    such that:
    \begin{equation}\label{eq:StrongSubConv}
      \lim_{j \rightarrow \infty}
      \sup_{t \in [0,\tau]} \|U^{(n_j)} - U\|^2
       + \int_0^\tau | A (U^{(n_j)} - U)|^2ds  = 0,
    \end{equation}
    almost surely.
  \item[(ii)] for any $p \geq 1$, there exists a sequence of
    $\Omega_{n_j} \in \mathcal{F}_0$, with $\Omega_{n_j} \uparrow
    \Omega$ such that:
   \begin{equation}\label{eq:UniformBnd}
      \sup_j \mathbb{E} \left[\indFn{\Omega_{n_j}} \left( \sup_{t \in
          [0,\tau]} \|U^{(n_j)}\|^2 + \int_0^\tau | A U^{(n_j)}|^2
        ds \right)^{p/2} \right]< \infty
    \end{equation}
    and
    \begin{equation}\label{eq:MeanBndOnCandidateSoln}
      \mathbb{E}\left(\sup_{t \in [0,\tau]} \|U\|^2
        + \int_0^\tau |AU|^2 ds \right)^{p/2} < \infty
    \end{equation}
  \end{itemize}
\end{Prop}

\begin{Rmk}\label{r3.1}
  The technical condition \eqref{eq:uniformDataBnd} is needed
  so that we may obtain the uniform pathwise bound:
  \begin{equation}\label{eq:UniformPathwiseBndStrayTm}
  \begin{split}
    \sup_{m,n}
    \underset{\omega \in \Omega}{\esssup } \left(
        \sup_{0 \leq t' \leq \tau^{M}_{m,n}}
    \|U^{(m)}\|^{2} + \int_{0}^{\tau^{M}_{m,n}}(1 + |AU^{(m)}|^{2})
        ds\right)
    &< \infty.\\
  \end{split}
  \end{equation}
  See \eqref{eq:UniformLocalExistTms},
  \eqref{eq:FinalInequalityConclusion} below.
  Note however that this condition may be removed
  in the final step of the proof of the local existence.  See
  \cite[Proposition 4.2]{GlattHoltzZiane2}.
\end{Rmk}

\begin{proof}
  As in previous work \cite{GlattHoltzZiane2}, the proof consists in
  establishing the sufficient conditions \eqref{eq:LocalCauchyCriteria},
  \eqref{eq:GrowthCond2} for \cite[Lemma
  5.1]{GlattHoltzZiane2} (see also related results in \cite{MikuleviciusRozovskii2}), 
  from which (i) and (ii) follow
  directly.  The proof makes use of some delicate estimates present 
  even in the deterministic case ($\sigma \equiv 0$) that
  have been carried out in a separate work \cite{GlattHoltzTemam1}.

  We assume with no loss of generality that $M > 1$ and consider the stopping times
  \begin{equation}\label{eq:UniformLocalExistTms}
    \tau^{M}_{n}
      = \inf_{t \geq 0} \left\{
       \sup_{t' \in [0,t]} \|U^{(n)}\|^2 + \int_0^t |A U^{(n)}|^2 dt'
          > 4 M
     \right\}.
  \end{equation}
Note that \eqref{eq:UniformLocalExistTms} implies that
\begin{equation}\label{e3.8a}
\sup_{t' \in [0,t]} ||U^{(n)}||^2 + \int^t_0 |AU^{(n)}|^2dt'\leq 4M,
\text{ for } 0\leq t < \tau^M_n.
\end{equation}
  We set $\tau^{M}_{m,n} := \tau^{M}_{n} \wedge \tau^{M}_{m}$.
  The first step in the proof is to perform estimates on
  $U^{(m)} - U^{(n)}$ which we denote by $R^{(m,n)}$ to simplify the
  notation below.  We will show that
\begin{equation}\label{eq:LocalCauchyCriteria}
    \lim_{n \rightarrow \infty} \sup_{m > n}
        \mathbb{E} \left( \sup_{0 \leq t' \leq \tau_{m,n}^M} \|R^{(m,n)} \|^2 +
          \int_0^{\tau_{m,n}^M }| A R^{(m,n)}|^2 dt \right) =0 ,
  \end{equation}
  which is the first condition required for \cite[Lemma 5.1]{GlattHoltzZiane2}.

  We fix $m > n$, subtract the equations for $m, n$, then apply
  $A^{1/2}$ to the resulting system.  Note that $D(A^{1/2}) = V$
  with $\|U\|^2 = |A^{1/2}U|$.  By the It\={o} lemma we may also infer that
  \begin{equation}\label{eq:diffH1Differential}
    \begin{split}
      d \| R^{(m,n)}\|^2
         + & 2| A R^{(m,n)}  |^2 dt \\
         = &- 2\langle P_{m} N(U^{(m)})  - P_{n} N(U^{(n)}),  A R^{(m,n)} \rangle dt\\
       &+ 2 \langle P_{m}^{n} F, A R^{(m,n)}\rangle dt\\
        &+ \| P_m\sigma(U^{(m)}) - P_n\sigma(U^{(n)}) \|^2_{L_2(\mathfrak{U},V)} dt\\
          &+ 2  \langle P_m \sigma(U^{(m)}) - P_n \sigma(U^{(n)}),  A R^{(m,n)}\rangle dW.\\
    \end{split}
  \end{equation}
  We now estimate each of the terms above with a view of finally
  applying a stochastic analogue of the Gronwall inequality,
  \cite[Lemma 5.3]{GlattHoltzZiane2}.   With this in mind fix any pair of
  stopping times $\tau_a, \tau_b$ such that $0
  \leq \tau_a \leq \tau_b \leq \tau^M_{n,m}$. By integrating the above system,
  taking a supremum over the random interval $[\tau_a, \tau_b]$
and finally taking an expected values we may infer that
  \begin{equation}\label{eq:mainInequality}
    \begin{split}
      \E \Biggl( \sup_{t \in [\tau_a, \tau_b]}&\| R^{(m,n)} \|^2
         + \int_{\tau_a}^{\tau_b}  | AR^{(m,n)} |^2 dt \Biggr)\\
         \leq&
         c\E \|R^{(m,n)} (\tau_a)\|^2
         + c \E \int_{\tau_a}^{\tau_b}
        |\langle (P_{m} - P_{n}) F,  A R^{(m,n)}  \rangle| dt\\
         &+  c \E \int_{\tau_a}^{\tau_b}
           |\langle P_{m}N(U^{(m)}) -P_{n}N(U^{(n)}), A R^{(m,n)}
           \rangle| dt\\
      &+ c\E \int_{\tau_a}^{\tau_b}\| P_m\sigma(U^{(m)}) 
           - P_n\sigma(U^{(n)}) \|^2_{L_2(\mathfrak{U},V)} dt\\
         &+ c \E \sup_{t \in [\tau_a,\tau_b]}
         \left|\int_{\tau_a}^t\langle P_m \sigma(U^{(m)}) - P_n
           \sigma(U^{(n)}),  A R^{(m,n)} \rangle dW\right|.\\
    \end{split}
  \end{equation}

  We begin by addressing the `deterministic portions' of
  \eqref{eq:mainInequality}.  Using the equivalence fractional order
  spaces, \eqref{eq:H1EstBU} and the generalized Poincar\'{e}
  inequality it is
  shown in \cite{GlattHoltzTemam1}, (see \eqref{eq:DetCauchyEstSummary} in Theorem 3.1 of \cite{GlattHoltzTemam1})
  that:
  \begin{equation}\label{eq:DetCauchyEstSummary}
    \begin{split}
        |\langle P_{m}N(U^{(m)}) -& P_{n}N(U^{(n)}),A R^{(m,n)}\rangle|\\
        \leq& \frac{1}{2} |A R^{(m,n)}|^{2} +
        c(1+ |AU^{(m)}|^2+ \|U^{(n)}\|^4) \|R^{(m,n)}\|^2\\
        &+\frac{c}{\lambda_n^{1/4}} (1+ \|U^{(n)}\|^2)
               (1 + |AU^{(n)}|^2).
       \end{split}
  \end{equation}
 We next consider the terms which arise only in the stochastic context.
 The It\={o} correction term may be estimated according to
 \begin{equation}\label{eq:itoCorectionComp}
   \begin{split}
   \| P_m \sigma(U^{(m)}) &-
       P_n \sigma(U^{(n)})\|^2_{L_2(\mathfrak{U}, V)}\\
       \leq& c \left(
          \| \sigma(U^{(m)}) -
        \sigma(U^{(n)})\|^2_{L_2(\mathfrak{U}, V)} +
        \| Q_n \sigma(U^{(n)}) \|^2_{L_2(\mathfrak{U}, V)} \right)\\
        \leq& c(\|R^{m,n}\|^2 + \frac{1}{\lambda_n}
        | A \sigma(U^{(n)}) |^2_{L_2(\mathfrak{U}, H)})\\
        \leq& c \left(\|R^{m,n}\|^2
          + \frac{1}{\lambda_n}(1 + |AU^{(n)}|^2)\right)
  \end{split}
\end{equation}
For the second inequality we have made use of the generalized Poincar\'{e} Inequality\footnote{We use the special case
$\|Q_{n}U^{\sharp}\|^{2} \leq \tfrac{1}{\lambda_{n}} |A U^{\sharp}|^{2}$, which holds for any $U^{\sharp} \in D(A)$.}.
The final inequality follows from  (\ref{eq:lipCondClass}).   For the stochastic integral terms we apply \eqref{eq:BDG} and deduce
\begin{equation}  \label{eq:stochasticIntCompEstBDG}
  \begin{split}
          \mathbb{E} & \sup_{\tau_a \leq t' \leq \tau_b}
        \left|\int_{\tau_a}^{t'} \langle P_m\sigma(U^{(m)}) - P_n \sigma(U^{(n)}) , A R^{(m,n)} \rangle
                 dW \right|\\
              \leq& c
      \mathbb{E}
       \left( \int_{\tau_a}^{\tau_b}
         \langle P_m\sigma(U^{(m)}) - P_n \sigma(U^{(n)}) , A R^{(m,n)}
                 \rangle_{L_2(\mathfrak{U}, H)}^2
                dt'\right)^{1/2}\\
            \leq& c
      \mathbb{E}
       \left( \int_{\tau_a}^{\tau_b}
           \| P_m\sigma(U^{(m)})
            - P_n \sigma(U^{(n)})\|^2_{L_2(\mathfrak{U}, V)}
           \| R^{(m,n)} \|^2
               dt'\right)^{1/2}\\
            \leq& c
      \mathbb{E}
       \left(\sup_{t \in [\tau_a, \tau_b]} \| R^{(m,n)} \| \right.\\
         &\quad \quad \quad\left. \cdot \left( \int_{\tau_a}^{\tau_b}
           \| P_m\sigma(U^{(m)})
            - P_n \sigma(U^{(n)})\|^2_{L_2(\mathfrak{U}, V)}
               dt'\right)^{1/2} \right)\\
            \leq& \frac{1}{2}
      \mathbb{E} \left(
        \sup_{t \in [\tau_a, \tau_b]} \| R^{(m,n)} \|^2 \right) \\
          &+ c \E
       \left( \int_{\tau_a}^{\tau_b} (
         \|R^{(m,n)}\|^2 + \frac{1}{\lambda_n}(1 + |AU^{(n)}|^2)
         	)
              dt'\right).\\
 \end{split}
\end{equation}
The last inequality is achieved by applying the Schwarz inequality and
then \eqref{eq:itoCorectionComp}.

We now gather the estimates \eqref{eq:DetCauchyEstSummary},
\eqref{eq:itoCorectionComp}), \eqref{eq:stochasticIntCompEstBDG}
and compare with \eqref{eq:mainInequality}.
Since $0 \leq \tau_a \leq \tau_b \leq \tau^M_{m,n}$ we conclude,
using \eqref{e3.8a} that
\begin{equation}\label{eq:FinalInequalityConclusion}
  \begin{split}
    \E \Biggl( &\sup_{t \in [\tau_a, \tau_b]}\| R^{(m,n)} \|^2
         + \int_{\tau_a}^{\tau_b}  | A R^{(m,n)} |^2 dt \Biggr)\\
       \leq& c\  \E \|R^{(m,n}(\tau_a) \|^2 \\
       &+c\ \E \int_{\tau_a}^{\tau_b}  \left(
              (1 + |AU^{(m)}|^2) \|R^{(m,n)} \|^2+
              \frac{1}{\lambda_n^{1/4}} (1+ |AU^{(n)}|^2) +
              | Q_n F|^2 \right)dt.
  \end{split}
\end{equation}
Observe that the generic constant $c$ is independent of $m,n$
and that \eqref{eq:UniformPathwiseBndStrayTm}
(\ref{eq:LocalCauchyCriteria}) now follows from the stochastic
Gronwall lemma.

It remains to establish the other requirement
of \cite[Lemma 5.1]{GlattHoltzZiane2}. In the present context this translates to
\begin{equation}\label{eq:GrowthCond2}
  \lim_{\delta \rightarrow 0} \sup_{n}
     \Prb \left( \sup_{0 \leq t' \leq \tau^M_n \wedge \delta} \|U^{(n)}\|^2
          + \int_0^{\tau^M_n \wedge \delta} | AU^{(n)}|^2 dt' > \tilde{M} \right)
          = 0,
\end{equation}
for every $\tilde{M} > M$.
By applying It\={o} we infer an equation for $t \mapsto
\|U^{(n)}(t)\|^2$ very similar to
(\ref{eq:AhalfUdifferential}), below. Since, as for the Navier-Stokes
system in $d = 3$ (see (\ref{eq:strongTypeEstimate}))
\begin{displaymath}
  |\langle B(U), AU \rangle| \leq \|U\|^{3/2} |AU|^{3/2},
  \quad U \in D(A)
\end{displaymath}
and since the $A_p$ and $E$ terms are lower order (see
\eqref{eq:estimateAloworder},\eqref{eq:CorbndOperator}), we may establish
(\ref{eq:GrowthCond2}) with a direct application of Doob's inequality
exactly as in \cite[Proposition 3.1]{GlattHoltzZiane2}.  With
\eqref{eq:GrowthCond2}, the proof is complete.
\end{proof}

\section{Global Existence}

We now implement a series of anisotropic estimates that are used
to infer global existence. Due to the non-commutativity
introduced by the physical boundary conditions we must first define
a new variable $\hat{U}$ that satisfies a system obeying the
rules of ordinary calculus.  We are then are able to derive suitable
estimates for $\hat{u}_z$ and then $\hat{u}_x$ and finally for the entire
original system in $V$. Since the
resulting estimates yield only pathwise (rather than moment) bounds
we must finally recourse to some involved stopping time arguments which make
essential use of Lemma~\ref{thm:stArg1}.

\subsection{A Change of Variable to a Random PDE and some Auxiliary Estimates}
\label{sec:change-vari-rand}

We consider the linear stochastic partial differential equation
\begin{subequations}\label{eq:PE2DLinearStoch}
  \begin{gather}
     \pd{t} \check{u} - \nu \Delta \check{u} + \pd{x} \check{p}_s
       =\indFn{t \leq \xi} \sigma_u(\mathbf{v}, T) \dot{W}_1,
    \label{eq:PEMoment1LS}\\
    \pd{t} \check{v} - \nu \Delta \check{v} =\indFn{t \leq \xi} \sigma_v(\mathbf{v},T) \dot{W}_2,
    \label{eq:PEMoment2LS}\\
    \pd{t}\check{T} - \mu \Delta \check{T} = \indFn{t \leq \xi} \sigma_T(\mathbf{v},T) \dot{W}_3,
    \label{eq:PETempCoupleLS}
  \end{gather}
\end{subequations}
with $\xi$ as in Proposition~\ref{thm:MaxExist}.  This system is supplemented with the same boundary
conditions as in \eqref{eq:PE2DBasicBC}. We posit the zero initial condition $\check{u}(0) =
\check{v}(0) = \check{T}(0) = 0$. Note that the stochastic forcing
terms depend on $(U,\xi) = ((\mathbf{v},T), \xi)$, maximal strong
solution solution we found for
(\ref{eq:PE2DBasic})-\eqref{eq:basicInitialCond} in
Proposition~\ref{thm:MaxExist}; $\sigma$ is exactly the same as
appearing in \eqref{eq:PE2DBasic} and in particular satisfies
\eqref{eq:lipCond}. As in Section~\ref{sec:definition-solutions},
\eqref{eq:PE2DLinearStoch} may be formulated in an abstract form:
\begin{equation}\label{eq:auxLinSystemAbs}
  d \check{U} + A \check{U} dt = \indFn{t \leq \xi}  \sigma(U)dW, \quad \check{U}(0) = 0.
\end{equation}
We shall need the following preliminary estimates  below for $\check{U}$.
\begin{Lem}\label{thm:AuxSystemEst}
  There exists a unique global pathwise strong solution of
  (\ref{eq:auxLinSystemAbs}) taking its values in $D(A)$.  Additionally for any
  deterministic finite time $t > 0$, we have
  \begin{equation}\label{eq:checkStrongEstimate}
    \mathbb{E} \left( \sup_{t' \in [0, t]} | A\check{U}|^2 \right) < \infty.
 \end{equation}
\end{Lem}
\begin{proof}
  We briefly outline the formal estimates that lead to (\ref{eq:checkStrongEstimate}).
  Since \eqref{eq:auxLinSystemAbs} is linear in the unknown
  everything, including the global existence, may be easily justified
  with a suitable Galerkin scheme (see e.g. \cite{Flandoli1}).

  Formally then we multiply \eqref{eq:auxLinSystemAbs} by $A$ and apply the
  It\={o} lemma in $H$ to deduce
  \begin{equation}\label{eq:H2EqnforcheckU}
    \begin{split}
    d | A \check{U}|^2 + 2 &|A^{3/2} \check{U}|^2 dt\\
    =& 2 \indFn{t \leq \xi}\langle  A \sigma(U), A \check{U} \rangle dW
      + \indFn{t \leq \xi} | A\sigma(U) |_{L_2(\mathfrak{U}, H)}^2 dt.\\
   \end{split}
 \end{equation}
  Fixing arbitrary $t > 0$ and taking a supremum over $t' \leq t$ and
  then expected values we infer from \eqref{eq:lipCondClass}, \eqref{eq:H2EqnforcheckU} and
  the fact that $\tilde{U}(0) = 0$,
  \begin{equation}\label{eq:tildeUverystongEstimates}
    \begin{split}
     \mathbb{E}& \left( \sup_{t' \in [0, t]}| A\check{U} |^2 \right)\\
     &\leq \mathbb{E} \sup_{t' \in [0, t]} \left| \int_0^{t' \wedge\xi}\langle
      A\sigma(U), A \check{U} \rangle dW \right|+
      \mathbb{E}\int_0^{t \wedge\xi}| \sigma(U) |_{L_2(\mathfrak{U},D(A))}^2 dt'\\
      &\leq \frac{1}{2} \mathbb{E} \left(
        \sup_{t' \in [0, t]}| A\check{U}|^2 \right)
                 + c \mathbb{E}\int_0^{t \wedge\xi}| \sigma(U)
                 |_{L_2(\mathfrak{U}, D(A)}^2 dt'\\
       &\leq \frac{1}{2} \mathbb{E} \left(
         \sup_{t' \in [0, t]}| A\check{U}|^2 \right)
       +c\mathbb{E} \int_0^{t \wedge \xi} (1+ \|U \|^2)  dt'.
   \end{split}
 \end{equation}
 For the stochastic integral terms after the first inequality we apply \eqref{eq:BDG}
 and then estimate in a similar manner to \eqref{eq:stochasticIntCompEstBDG}.
 The final inequality is a consequence of the assumption \eqref{eq:lipCond} imposed on $\sigma$.
 To complete the proof we rearrange \eqref{eq:tildeUverystongEstimates} and refer to \eqref{eq:weakBnds} in
 Proposition~\ref{thm:MaxExist} to conclude
 \eqref{eq:checkStrongEstimate}.
\end{proof}

We next subtract \eqref{eq:PEMoment1LS} from \eqref{eq:PE2Dreform} and
define $\hat{U} = U - \check{U}$.  On the random interval $[0,
\xi)$ we see that $\hat{U}$ must
satisfy the following partial differential equation (without white
noise driven forcing but with random coefficients)
\begin{equation}\label{eq:UPDEforzEstAbsPre}
  \frac{d}{dt} \hat{U} + A \hat{U} + A_p (\hat{U} + \check{U})
                                 + B(\hat{U} + \check{U})
                                 + E(\hat{U} + \check{U}) = F.
\end{equation}
Note that, in contrast to \eqref{eq:mainSystemAbsDiffForm} this new
system satisfies the usual rules of ordinary calculus.

We may rewrite \eqref{eq:UPDEforzEstAbsPre} in a form more convenient for our purposes below:
\begin{equation}\label{eq:UPDEforzEstAbs}
  \begin{split}
  \frac{d}{dt} \hat{U} + A \hat{U} &+ A_p \hat{U}
                                 + B(\hat{U})
                                 + E \hat{U} \\
                                 =& F
                                 - B(\check{U},\check{U})
                                 - B(\check{U},\hat{U})
                                 - B(\hat{U},\check{U})
                                 - E\check{U} - A_p \check{U}.
 \end{split}
\end{equation}
By combining Lemma~\ref{thm:AuxSystemEst} with Proposition~\ref{thm:MaxExist} we may directly infer that
\begin{Lem}\label{thm:prelimestRandomSystem}
  For any deterministic, finite $t > 0$ we have:
  \begin{equation}\label{eq:weakEstUcheck}
    \mathbb{E} \left( \sup_{0 \leq t' \leq \xi \wedge t} | \hat{U}|^2
    + \int_0^{\xi \wedge t}  \|\hat{U} \|^2 ds\right) < \infty
  \end{equation}
\end{Lem}

Finally we note that the first momentum equation included in
(\ref{eq:UPDEforzEstAbs}), which will be the focus of our attention in
the subsequent sections, is given by
\begin{equation}\label{eq:firstMomentumMinusNoise}
  \begin{split}
      \pd{t} \hat{u}
      &+ \hat{u} \pd{x} \hat{u} + w(\hat{u}) \pd{z} \hat{u}
      - \nu \Delta \hat{u} - f \hat{v}
      + \pd{x} \hat{p}_s
      - \beta_T g \rho_0 \int_z^0  \pd{x} \hat{T} d \bar{z}  \\
    =& F_u + f \check{v}
        +\beta_T g \rho_0 \int_z^0  \pd{x} \check{T} d \bar{z}  \\
    &-( \check{u} \pd{x} \check{u} + w(\check{u}) \pd{z} \check{u} )
    -(  \check{u} \pd{x} \hat{u} + w(\check{u}) \pd{z} \hat{u} )
    -(  \hat{u} \pd{x} \check{u} + w(\hat{u}) \pd{z} \check{u} )\\
    =& F_u + f \check{v}
    +\beta_T g \rho_0 \int_z^0  \pd{x} \check{T}d \bar{z}
    - (\tilde{\mathcal{B}}^1(\check{u}, \check{u}) +
        \tilde{\mathcal{B}}^1(\check{u},\hat{u}) +
        \tilde{\mathcal{B}}^1(\hat{u}, \check{u})).
\end{split}
\end{equation}

\begin{Rmk}\label{r4.1}
We infer from (\ref{eq:weakEstUcheck}) that,
\begin{equation}\label{e4.9a}
\sup_{0\leq t'\leq\xi\wedge t} |\hat{U}|^2 + \int^{\xi\wedge t}_0
||\hat{U}||^2 ds \leq K^1(t,\omega) < \infty
\end{equation}
where here and below, $K, K^i,$ denote a.s. finite constants which depend
on $t$, on the data such as norms of $U_0$, $F$ and on $\omega$ 
through these norms and though stochastic integral terms
driven by $W$.
\end{Rmk}

\subsection{Anisotropic Estimates}
\label{sec:time-unif-estim-uz}
We now turn to the estimates for $\pd{z} \hat{u}$.
\begin{Lem}\label{thm:uzestimatesLemma}
  Let $(U,\xi) = ((u,v,T), \xi)$ be the unique maximal strong
  solution of \eqref{eq:mainSystemAbsDiffForm} guaranteed by Proposition~\ref{thm:MaxExist}.
Then, for every $t >0$ there exists a finite constant $K=K^2(t,\omega) <
\infty$ depending on $t,\omega$ and the data such that
  \begin{equation}\label{e4.10b}
     \sup_{0\leq t'\leq\xi\wedge t} |\pd{z} \hat{u} |^2
      + \int^{\xi\wedge t}_0\| \pd{z} \hat{u} \|^2ds
      \leq K^2  \quad a.s.
 \end{equation}
\end{Lem}
\begin{proof}
  We multiply \eqref{eq:firstMomentumMinusNoise} by $-Q \pd{zz}\hat{u}$ and
  integrate over the domain $\mathcal{M}$.  Following closely the computations
  in \cite{PetcuTemamZiane} we to deduce:
  \begin{equation}\label{eq:QuzzevolutionEqnRandomPDE}
    \begin{split}
      \frac{1}{2}  \frac{d}{dt}
      \bigl( |\pd{z}\hat{u}|^2
        &+ \alpha_{\mathbf{v}} | \hat{u} |^2_{L^2(\Gamma_i)}\bigr)
           + \nu \| \pd{z}\hat{u} \|^2
           + \nu \alpha_{\mathbf{v}} |\pd{x}\hat{u}|^2_{L^2(\Gamma_i)}\\
      =& |P \pd{zz} \hat{u}|^2
           - \int_{\mathcal{M}}F_uQ \pd{zz}\hat{u} \, d\mathcal{M}\\
        &- \beta_T g \rho_0
              \int_{\mathcal{M}} \left( \int^0_z
              \pd{x}(\hat{T} + \check{T})d\bar{z} \right)Q\pd{zz}\hat{u} \, d\mathcal{M}\\
        &-  \int_{\mathcal{M}}f(\hat{v} +\check{v})Q \pd{zz}\hat{u} \, d\mathcal{M}\\
        &+ \frac{2}{h} \int_{\mathcal{M}} \hat{u} \pd{x}\hat{u}
              \left[\alpha_{\mathbf{v}} \hat{u}(0,x) + \pd{z}\hat{u}(x,-h) \right]
              \, d \mathcal{M}\\
        &+ \int_{\mathcal{M}} ( B^1(\check{u}, \check{u})
                  + B^1(\check{u},\hat{u})
                  + B^1(\hat{u}, \check{u}))\pd{zz}\hat{u}
                d\mathcal{M}\\
       &= J_1 + J_2 + J_3  + J_4 + J_5 + J_6 + J_7 + J_8.
   \end{split}
\end{equation}
Here the bottom
boundary is flat which causes several terms to disappear
present in \cite{PetcuTemamZiane}.  The term
\begin{displaymath}
  -\langle B^1(\hat{u}, \hat{u}), - \pd{zz} \hat{u} \rangle
  = \int (\hat{u} \pd{x} \hat{u} + w( \hat{u}) \pd{z} \hat{u})
  \mathit{Q} \pd{zz}\hat{u} \, d \mathcal{M}
\end{displaymath}
largely cancels and appears as $J_5$ due to Lemma~\ref{thm:Best}, (vi)
above.  Also we observe that $Q \pd{x} \hat{p}_s =0$ which is why we
multiply (\ref{eq:firstMomentumMinusNoise}) by $-Q\pd{zz}\hat{u}$
rather than $-\pd{zz} u$.

The first term $J_{1}$ on the right hand side of \eqref{eq:QuzzevolutionEqnRandomPDE}
reduces to two terms at $z=-h$ and $0$ that are estimated using the trace theorem:
\begin{equation}\label{eq:J1Estdz}
    |J_1| \leq
    c \|\hat{U} \|^2+
    \frac{\nu}{16} \|\pd{z}\hat{u} \|^2.
\end{equation}
The estimates for the next three terms are direct:
\begin{equation}\label{eq:J1234Estdz}
  \begin{split}
    |J_2| \leq&
    c |F|^2 +
    \frac{\nu}{16} \|\pd{z}\hat{u}\|^2,\\
    |J_3| \leq &
    c(\|\hat{U}\|^2 + \|\check{U}\|^2)  +
   \frac{\nu}{16} \|\pd{z}\hat{u} \|^2\\
   \leq &
   c(\|\hat{U}\|^2 + |\check{U}|^2_{(2)})  +
  \frac{\nu}{16} \|\pd{z}\hat{u} \|^2,\\
    |J_4| \leq &
    c(\|\hat{U}\|^2 + \|\check{U}\|^2) +
    \frac{\nu}{16} \|\pd{z}\hat{u} \|^2\\
    \leq &
   c(\|\hat{U}\|^2 + |\check{U}|^2_{(2)})  +
  \frac{\nu}{16} \|\pd{z}\hat{u} \|^2.\\
  \end{split}
\end{equation}
For $J_5$ we may estimate using
\eqref{eq:semiCancelpdzEst} and Young's inequality
\begin{equation}\label{eq:J5Estdz}
  \begin{split}
  |J_5| \leq&
  c (|\hat{u}| \| \hat{u} \|^2
             +|\pd{z}\hat{u}|^{1/2} \|\pd{z}\hat{u}\|^{1/2}
               |\hat{u}|^{1/2} \|\hat{u} \|^{3/2})\\
      \leq&
        c (|\hat{U}| \|\hat{U}\|^2 +
        |\pd{z}\hat{u}|^{2/3} |\hat{U}|^{2/3}
        \|\hat{U}\|^2)
        + \frac{\nu}{16} \|\pd{z}\hat{u} \|^2\\
      \leq&
        c (|\hat{U}| \|\hat{U}\|^2 +
        |\pd{z}\hat{u}|^{2}
        \|\hat{U}\|^2)
        + \frac{\nu}{16} \|\pd{z}\hat{u} \|^2.\\
\end{split}
\end{equation}
For $J_6$,  \eqref{eq:genericFirstCompEstL2ClassComp1},
\eqref{eq:genericFirstCompEstL2FuckedComp2} allow
\begin{equation}\label{eq:checkchecknonlinearTerm}
  \begin{split}
  |J_6| \leq&
     c (|\check{u}|^{1/2} \| \check{u} \| + \| \check{u} \|^{3/2}) | \check{u} |^{1/2}_{(2)}
          | \pd{zz} \hat{u} |\\
          \leq&
             c \|\check{u}\| | \check{u} |_{(2)} | \pd{zz} \hat{u} |\\
          \leq&
     c \| \check{U} \|^2 |\check{U}|^2_{(2)} + \frac{\nu}{16} \| \pd{z} \hat{u} \|^2.
  \end{split}
\end{equation}
For $J_7$ we estimate with \eqref{eq:genericFirstCompEstL2ClassComp1} and
\eqref{eq:genericFirstCompEstL2FuckedComp2}:
\begin{equation}\label{eq:J7uzzEasyPart}
  \begin{split}
    |J_{7}| \leq&
       c (|\check{u} |^{1/2} |\check{u} |_{(2)}^{1/2} \| \hat{u}\|
          +\|\check{u} \|^{1/2} |\check{u} |_{(2)}^{1/2}
          |\pd{z} \hat{u}|) | \pd{zz} \hat{u} |\\
      \leq&c |\check{U} |_{(2)}^2 (\| \hat{U} \|^2 + |\pd{z} \hat{u}|^2)
         +\frac{\nu}{16} \| \pd{z} \hat{u} \|^2.\\
  \end{split}
\end{equation}
Finally concerning $J_8 = \langle B^1_1(\hat{u}, \check{u}) +B^1_2(\hat{u}, \check{u}),\pd{zz}\hat{u} \rangle := J_{8,1} + J_{8,2} $ we estimate
\begin{equation}\label{eq:termJ8uzzEasyPart}
  \begin{split}
  |J_{8,1}|
      \leq& c | \hat{u} |^{1/2} \|\hat{u} \|^{1/2}
                \|\check{u} \|^{1/2} |\check{u} |_{(2)}^{1/2}
                \|\pd{z}\hat{u} \|\\
      \leq& c (| \hat{U} |^{2} \|\hat{U} \|^{2} +
                \|\check{U} \|^{2} |\check{U} |_{(2)}^{2} ) +
                \frac{\nu}{32}\|\pd{z}\hat{u} \|^2,\\
 \end{split}
\end{equation}
using \eqref{eq:genericFirstCompEstL2ClassComp2}, and,
\begin{equation}\label{eq:J82BIGPROBLEMHOUSTON}
  \begin{split}
    |J_{8,2}| &\leq c \| \hat{u} \| \| \check{u} \|^{1/2}
                            | \check{u} |_{(2)}^{1/2} \|\pd{z}\hat{u}\|\\
                 &\leq c \|\hat{U}\|^{2}  | \check{U} |_{(2)}^2
                  +\frac{\nu}{32} \|\pd{z}\hat{u}\|^2.\\
 \end{split}
\end{equation}
thanks to (\ref{eq:genericFirstCompEstL2FuckedComp1}).

Collecting the estimates \eqref{eq:J1Estdz}, \eqref{eq:J1234Estdz}, \eqref{eq:J5Estdz},
\eqref{eq:checkchecknonlinearTerm},  \eqref{eq:J7uzzEasyPart},
\eqref{eq:termJ8uzzEasyPart}, \eqref{eq:J82BIGPROBLEMHOUSTON}
above we may finally observe that
\begin{equation}\label{eq:diffinequalityduz}
  \begin{split}
\frac{d}{dt}
      \bigl( |\pd{z}&\hat{u}|^2
        + \alpha_{\mathbf{v}} | \hat{u} |^2_{L^2(\Gamma_i)}\bigr)
           + \nu \| \pd{z}\hat{u} \|^2\\
        \leq&
        c( \| \hat{U} \|^2  + |\check{U}|^2_{(2)} ) |\pd{z} \hat{u}|^2\\
        &+c (1+ |\hat{U}|^2)\| \hat{U}\|^2  +
                     c(1+ \|\hat{U} \|^2 + \| \check{U} \|^2 )
                     |\check{U} |^2_{(2)}
                 +c|F|^2.\\
  \end{split}
\end{equation}
We therefore conclude that
\begin{equation}\label{eq:gronwallsetUpdUZ}
  \begin{split}
  \frac{d}{dt}
      \bigl( |\pd{z}\hat{u}|^2
        + \alpha_{\mathbf{v}} | \hat{u} |^2_{L^2(\Gamma_i)}\bigr)
        \leq
        (|\pd{z} \hat{u}|^2 + \alpha_{\mathbf{v}} | \hat{u} |^2_{L^2(\Gamma_i)}) R_1 + R_2 + C|F|^2,
 \end{split}
\end{equation}
where
\begin{equation}\label{eq:GronwallTermsDef}
  \begin{split}
    R_1 &:= \| \hat{U} \|^2  + |\check{U}|^2_{(2)}\\
    R_2 &:= c(1+ |\hat{U}|^2)\| \hat{U}\|^2
         + c(1+ \|\hat{U} \|^2 + \| \check{U}\|^2 ) |\check{U} |^2_{(2)}
 \end{split}
\end{equation}
and the constants $c$ are as in \eqref{eq:diffinequalityduz}.
Note that, due to \eqref{e4.9a} and \eqref{eq:checkStrongEstimate}, for all $t > 0$,
there exists a constant $K=K(t,\omega)$ such that,
\begin{equation}\label{eq:PathwiseBndsR1R2pdz}
    \int_{0}^{t\wedge\xi} R_{j} ds \leq K(t,\omega) < \infty \;\; a.s.
    \quad j = 1,2.
\end{equation}
The (deterministic) Gronwall inequality now yields
\begin{equation}\label{eq:GronwallInequalUzzEst}
  \begin{split}
  \sup_{t' \in [0, \tau_n \wedge t]} |\pd{z} \hat{u}|^2
    &\leq   \sup_{t' \in [0, \tau_n \wedge t]} (|\pd{z} \hat{u}|^2 +
        \alpha_{\mathbf{v}} | \hat{u} |^2_{L^2(\Gamma_i)})\\
    &\leq \exp \left(\int_0^{\xi\wedge t} R_1 dt' \right)
       \left(  | \pd{z} u_0|^2 + \int_0^{\xi\wedge t} (R_2 + C|F|^2 )dt' \right)\\
    & \leq K(t, \omega) \left( 1 + \|U_0\|^2 + \int_0^{\xi\wedge t} |F|^2 dt' \right).
 \end{split}
\end{equation}
Finally, returning to (\ref{eq:diffinequalityduz}), integrating over $[0,
\tau_n \wedge t]$, and
then neglecting the terms  $|\pd{z}\hat{u}|^2 +
\alpha_{\mathbf{v}} | \hat{u} |^2_{L^2(\Gamma_i)}$ appearing on the
left hand side of the resulting expression,  we observe that:
\begin{equation}\label{eq:tmIntViscosityTrmUz}
  \begin{split}
  \int_0^{\xi\wedge t} \| \pd{z} \hat{u} \|^2 dt'
     \leq& \|U_0\|^2 +  \int_0^{\xi\wedge t} (|\pd{z} \hat{u}|^2 R_1 + R_2
     + c|F|^2 )dt'\\
     \leq& K(t,\omega).
 \end{split}
\end{equation}
Combining (\ref{eq:GronwallInequalUzzEst}) and
(\ref{eq:tmIntViscosityTrmUz}), completes the proof.
\end{proof}

We next come to the estimates for $\pd{x}u$.    Here we show
\begin{Lem}\label{thm:uxestimatesLemma}
The hypotheses are the same as in Lemma \ref{thm:uzestimatesLemma}.
Then, for every $t > 0$, there exists a finite constant $K=K^3(t,\omega) <
\infty$  depending on $t,\omega$ and the data such that
  \begin{equation}\label{e4.24b}
      \sup_{0 \leq t' \leq \xi\wedge t} |\pd{x} \hat{u} |^2
      + \int_0^{\xi\wedge t} \|\pd{x} \hat{u}\|^2 dt'
      \leq K^3 \quad a.s.
 \end{equation}
\end{Lem}
\begin{proof}
The hypotheses being the same as for Lemma 4.3, the conclusions of
that Lemma thus hold, and in particular (\ref{e4.10b}).

To determine an evolution equation for $|\pd{x} \hat{u}|$ we multiply
  (\ref{eq:firstMomentumMinusNoise}) by $-\pd{xx} u$ and integrate over $\mathcal{M}.$
  After some direct manipulations, this yields
  \begin{equation}\label{eq:evolutionequationuxhat}
   \begin{split}
     \frac{1}{2} \frac{d}{dt} | &\pd{x} \hat{u}|^2
       + \nu \| \pd{x} \hat{u}\|^2
         + \nu \alpha_{\mathbf{v}} | \pd{x} \hat{u} |^2_{L^2(\Gamma_i)}\\
          =& \beta_T g \rho_0
           \int_{\mathcal{M}} \left(
           \int^0_z \pd{x}(\hat{T} + \check{T}) d\bar{z}  \right)
           \pd{xx}\hat{u} \, d\mathcal{M}\\
     &-  \int_{\mathcal{M}}F_u\pd{xx}\hat{u} \, d\mathcal{M}\\
    &-  \int_{\mathcal{M}}2f(\hat{v} +\check{v}) \pd{xx}\hat{u} \, d\mathcal{M}\\
     &+ \int_{\mathcal{M}}  (B^1(\hat{u}, \hat{u}) +
            B^1(\check{u},\check{u})+
            B^1(\check{u},\hat{u}) +
            B^1(\hat{u}, \check{u}))
     \pd{xx}\hat{u} \, d\mathcal{M}\\
       =& J_1 + J_2 + J_3 + J_4 + J_5 + J_6 + J_7.
      \end{split}
   \end{equation}
   Notice that in this case the pressure term disappears by
   integration in $z$, since $P
   \pd{xx} \hat{u} = 0$

   As above the first three terms are direct
   \begin{equation}\label{eq:J1234Estdx}
     \begin{split}
       |J_1| \leq&
       c(\|\hat{U}\|^2 + \|\check{U}\|^2) +
       \frac{\nu}{14} \|\pd{x}\hat{u} \|^2,\\
       |J_2| \leq&
       c |F|^2 +
       \frac{\nu}{14} \| \pd{x}\hat{u} \|^2,\\
       |J_3| \leq &
       c (\|\hat{U}\|^2 + \|\check{U}\|^2)  +
       \frac{\nu}{14} \| \pd{x}\hat{u} \|^2.\\
     \end{split}
   \end{equation}
   We may handle the term $J_4$ as in \cite{TemamZiane1}, however we
   may also directly apply Lemma~\ref{thm:Best},
   (\ref{eq:genericFirstCompEstL2ClassComp2}),
   (\ref{eq:genericFirstCompEstL2FuckedComp1}) to infer
   \begin{equation}\label{eq:nonlinear1Bux}
     \begin{split}
       |J_4|
          \leq&
          c ( |\hat{u}|^{1/2} \| \hat{u} \|^{1/2}
             |\pd{x}\hat{u}|^{1/2} \| \pd{x}\hat{u} \|^{3/2}
              +  |\pd{x} \hat{u} |
                    |\pd{z} \hat{u} |^{1/2}
                    \|\pd{z} \hat{u} \|^{1/2}
                    \|\pd{x}\hat{u}\| )\\
          \leq&c(|\hat{u}|^{2} \|\hat{u} \|^{2}
                       |\pd{x}\hat{u}|^{2} +
                       |\pd{x} \hat{u} |^2
                       |\pd{z} \hat{u} |
                       \|\pd{z} \hat{u} \|)
                    + \frac{\nu}{14} \| \pd{x}\hat{u} \|^2\\
          \leq&c(|\hat{U}|^2 \|\hat{U} \|^2 +
                    \|\pd{z} \hat{u} \|^2)|\pd{x} \hat{u} |^2
                   + \frac{\nu}{14} \| \pd{x}\hat{u} \|^2.\\
 \end{split}
  \end{equation}
  The estimates   (\ref{eq:genericFirstCompEstL2ClassComp1}) -
   (\ref{eq:genericFirstCompEstL2FuckedComp2}) allow us to treat the
  remaining terms $J_5, J_6, J_7$ as well.  Indeed
  \begin{equation}\label{eq:nolinearuxJ5}
    \begin{split}
      |J_5| \leq&
      c ( | \check{u} |^{1/2}
      | \check{u} |_{(2)}^{1/2}
      \| \check{u} \|
     +
       \| \check{u} \|^{3/2}
       |\check{u} |_{(2)}^{1/2})
     \| \pd{x}\hat{u} \|\\
     \leq&c \| \check{U} \|^2 |\check{U}|^2_{(2)}
                + \frac{\nu}{14} \| \pd{x}\hat{u} \|^2.\\
  \end{split}
 \end{equation}
 Also
 \begin{equation}\label{eq:estimateuxJ6}
   \begin{split}
     | J_6| \leq& c
         ( |\check{u} |^{1/2} | \check{u} |_{(2)}^{1/2}
          |\pd{x} \hat{u} |
         +
          \|\check{u} \|^{1/2} | \check{u} |_{(2)}^{1/2}
          |\pd{z} \hat{u} | ) \| \pd{x}\hat{u} \|\\
          \leq&c  |\check{U}|_{(2)}^2 (|\pd{x} \hat{u} |^2 + |\pd{z} \hat{u} |^2)
                     + \frac{\nu}{14} \| \pd{x}\hat{u} \|^2\\
     \leq&c \| \hat{U} \|^2 |\check{U}|^2_{(2)}
               + \frac{\nu}{14} \| \pd{x}\hat{u} \|^2.\\
   \end{split}
 \end{equation}
 Finally
 \begin{equation}\label{eq:estimateuxJ7}
   \begin{split}
     |J_7| &\leq c
       (  | \hat{u} |^{1/2} \| \hat{u} \|^{1/2}
        \| \check{u} \|^{1/2}  | \check{u} |_{(2)}^{1/2} +
          | \pd{x} \hat{u} |
         \| \check{u} \|^{1/2} | \check{u} |_{(2)}^{1/2})
         \| \pd{x} \hat{u} \|\\
         &\leq
         c \| \check{u} \| | \check{u} |_{(2)}
              (| \hat{u} | \| \hat{u} \| + | \pd{x} \hat{u} |^2)
              + \frac{\nu}{14} \| \pd{x} \hat{u} \|^2\\
         &\leq
         c \| \hat{U} \|^2 | \check{U} |_{(2)}^2
              + \frac{\nu}{14} \| \pd{x} \hat{u} \|^2.\\
  \end{split}
 \end{equation}
Gathering the estimates above, we conclude that:
\begin{equation}\label{eq:estuxsummary}
  \begin{split}
    \frac{d}{dt} | \pd{x} \hat{u}|^2 &+ \nu \| \pd{x} \hat{u}\|^2\\
       \leq& c(|\hat{U}|^2 \|\hat{U} \|^2 +
                    \|\pd{z} \hat{u} \|^2) | \pd{x} \hat{u} |^2 \\
              & +c(\|\hat{U}\|^2 + \|\check{U}\|^2  + \| \check{U} \|^2 |\check{U}|^2_{(2)}
                        +\| \hat{U} \|^2 |\check{U}|^2_{(2)} )
                 + c |F|^2\\
        \leq& R_3 | \pd{x} \hat{u} |^2+ R_4 + c|F|^2,
 \end{split}
\end{equation}
where $R_3 := c(|\hat{U}|^2 \|\hat{U} \|^2 + \|\pd{z} \hat{u} \|^2)$
and $R_4 := c(\|\hat{U}\|^2 + \|\check{U}\|^2  + \| \check{U} \|^2
|\check{U}|^2_{(2)} +\| \hat{U} \|^2 |\check{U}|^2_{(2)} )$.
Dropping the term $\nu \| \pd{x} \hat{u} \|^2$, applying the
Gronwall inequality and then making use of the assumed bound
\eqref{e4.24b} we infer, using \eqref{e4.9a},
\eqref{eq:checkStrongEstimate} and \eqref{e4.10b}, that
\begin{equation}\label{eq:uniformbnduxl2}
  \begin{split}
    \sup_{0\leq t' \leq \xi\wedge t} |\pd{x} \hat{u}|^2
          \leq & \exp\left( \int_0^{\xi\wedge t} R_3 dt'\right)
           \left( |\pd{x} u_0 |^2 + \int_0^{\xi\wedge t} (R_4 + C|F|^2)dt'\right) \\
          \leq & K(t,\omega) < + \infty .
 \end{split}
\end{equation}
We then integrate \eqref{eq:estuxsummary} from $0, \xi\wedge t$ and
infer, using again \eqref{e4.9a}, \eqref{eq:checkStrongEstimate} and
\eqref{e4.10b}, that
\begin{equation}\label{eq:uniformBnduxH1}
  \begin{split}
    \int_0^{\xi\wedge t} \|\pd{x}\hat{u}\|^2 dt'
       &\leq \|U_0\|^2 +
             \int_0^{\xi \wedge t} (
             R_3 |\pd{x} \hat{u}|^2 + R_4 + |F|^2 )dt'\\
       &\leq K(t,\omega) < \infty ,
     \end{split}
\end{equation}
where the final inequality follows from the previous bound
\eqref{eq:uniformbnduxl2}.  This completes the proof of
Lemma~\ref{thm:uxestimatesLemma}.
\end{proof}

\begin{Rmk} \label{rmk:BndMods}
  With some minor modifications to the proof,
  Lemma~\ref{thm:uxestimatesLemma} may
  established if we merely assume that,
  \begin{equation}\label{eq:ptwiseControlBndsMod}
    \begin{split}
              \sup_{t' \leq \tau_n} \left( | \hat{U} |^2 + \| \check{U} \|^2 
               + |\pd{z} \hat{u} |^2 \right) + \int_0^{\tau_n} (\|\hat{U} \|^2 + 
                      |\check{U}|_{(2)}^2+
                    \|\pd{z} \hat{u} \|^2) dt'  \leq K < \infty \quad a.s.
   \end{split}
  \end{equation}
  On the other hand the proof of Lemma~\ref{thm:uzestimatesLemma} seems to require
  that 
  \begin{equation}\label{eq:StrongerCondOnUCheck}
     \sup_{t' \leq \xi}| \check{U}|_{(2)}^2 \leq K  < \infty
  \end{equation}
  This condition 
  is needed to order handle $J_{8}$ appearing in \eqref{eq:QuzzevolutionEqnRandomPDE}. 
  The requirement \eqref{eq:StrongerCondOnUCheck} is achieved due to
  \eqref{eq:checkStrongEstimate} but at the cost of a slightly more
  restrictive condition on $\sigma$, \eqref{eq:lipCond}, as
  compared to previous work.   We underline here that this is the only point in this work
  where we require the final condition in (\ref{eq:lipCond}).
 \end{Rmk}

\begin{Rmk}\label{r4.3}
We observe that the $H^1$ -norm $||\varphi||$ of a function
$\varphi$ is equivalent to the norm $(|\varphi|^2 +
|\partial_x\varphi|^2 + |\partial_z\varphi|^2)^{1/2},$ and the $H^2$
- norm $|\varphi|_{(2)}$ of $\varphi$ is equivalent to the norm
$(\|\partial_x\varphi \|^2 + \| \partial_z\varphi \|^2 + ||\varphi||^2)^{1/2}.$  We then infer from \eqref{e4.9a} and
Lemmas \ref{thm:uzestimatesLemma} and \ref{thm:uxestimatesLemma}, that $\hat{u}$ being
as in these lemmas, that for every $t>0,$ there exists
a constant $K=K^4(t,\omega)$ depending on $t,\omega$ and the data, such that
\begin{equation}\label{e4.37}
\sup_{0\leq t'\leq\xi\wedge t} ||\hat{u}||^2 + \int^{\xi\wedge t}_0 |\hat{u}|^2_{(2)} ds\leq K^4 < \infty\quad a.s.
\end{equation}

\end{Rmk}

\subsection{Strong estimates for U}
\label{sec:strong-type-estim}
With the above preliminaries now in hand we may now proceed to study
$U$ in the strong norms, the final step of the proof of global existence.
\begin{Lem}\label{thm:StrongEstU}
  Suppose that $0 < n < \infty$ is a deterministic constant and let $\tau_n \leq \xi$ be the
  stopping time defined by
  \begin{equation}\label{eq:controlBoundStrongFirstComp}
    \tau_n = \inf \left\{ t \geq 0: \int_0^{\xi\wedge t} |u|_{(2)}^2 dt' > n\right\} \wedge\xi .
  \end{equation}
  Then, for any $t > 0$ there exists a deterministic constant $K=K^5_n(t)$ depending on $n$,$t$ and the data,
  such that:
  \begin{equation}\label{eq:meanboundWControl}
    \E \left(
    \sup_{0\leq t' \leq \tau_n\wedge t} \| U \|^2
        + \int_{0}^{\tau_n\wedge t} |A U |^2 dt'
    \right)\leq
    K^5_n(t).
  \end{equation}
\end{Lem}
\begin{proof}
  By the It\={o} formula and truncation argument (see
  \cite{Bensoussan1}) we derive an equation for $t \mapsto \|U(t)\|$:
 \begin{equation}\label{eq:AhalfUdifferential}
    \begin{split}
      d \|U\|^2 +& 2|AU|^2dt \\
             =& (2\langle F -  A_pU - B(U) - EU, AU \rangle
             + \|\sigma(U)\|^2_{L_2(\mathfrak{U}, V)} )dt\\
               &+ 2\langle A^{1/2}\sigma(U) , A^{1/2} U \rangle dW.
    \end{split}
  \end{equation}
  Note that due to Proposition~\ref{thm:MaxExist} this equality holds on the
  interval $[0, \xi)$.

  Fix arbitrary stopping times $0 \leq \tau_a \leq \tau_b \leq \tau_n
  \wedge t$.
  We now make estimates of \eqref{eq:AhalfUdifferential} on this interval
  in order to apply the stochastic version of the Gronwall lemma in \cite[Lemma 5.3]{GlattHoltzZiane2}.  As typical,
  the stochastic terms are majorized by applying the
  Burkholder-Davis-Gundy inequality \eqref{eq:BDG},
  \begin{displaymath}
    \begin{split}
      \mathbb{E} \sup_{\tau_a \leq t' \leq \tau_b} &
        \left|\int_{\tau_a}^{t'} \langle A^{1/2}\sigma(U) , A^{1/2} U \rangle
                 dW \right|\\
              &\leq c\ \mathbb{E}
       \left( \int_{\tau_a}^{\tau_b}
         \langle A^{1/2}\sigma(U) , A^{1/2} U
                 \rangle_{L_2(\mathfrak{U}, H)}^2
                dt'\right)^{1/2}\\
           &\leq \frac{1}{2} \mathbb{E} \left(
                   \sup_{\tau_a \leq t' \leq \tau_b } \|U\|^2 \right)
                       + c\ \mathbb{E} \int_{\tau_a}^{\tau_b} (1 + \|U\|^2)ds.
    \end{split}
  \end{displaymath}
  By applying \eqref{eq:strongTypeEstimateFirstCompControl} we may estimate the
  nonlinear part of the equation
  \begin{displaymath}
    |\langle B(U), AU \rangle|
      \leq c \|u\|^{1/2} |u|^{1/2}_{(2)} \|U\| |AU|
      \leq c |u|_{(2)}^2 \|U\|^2  +  \frac{1}{4}|AU|^2
  \end{displaymath}
  Making use of these two observations and obvious applications of
  Young's inequality for the lower order terms (see
  \eqref{eq:estimateAloworder}, \eqref{eq:CorbndOperator} ) we may estimate
  \begin{equation}\label{eq:FinalEstU}
    \begin{split}
      \mathbb{E}& \left(
        \sup_{ \tau_a \leq t' \leq \tau_b} \|U\|^2
                       + \int_{\tau_a}^{\tau_b} |AU|^2 dt'
        \right)\\
        &\leq c\ \E \| U(\tau_a)\|^2 + c\ \E
        \int_{\tau_a}^{\tau_b} (1+|F|^2 + (1 + |u|_{(2)}^2) \|U\|^2)dt'.
    \end{split}
  \end{equation}
The Gronwall lemma in \cite{GlattHoltzZiane2} applies to real valued, non-negative processes $X,Y,Z,R$ defined on
an interval of time $[0,T),$ and such that, for a stopping time $0<\tau <T,$

\begin{equation*}
\mathbb{E}\int^\tau_0 (RX + Z) ds < \infty,
\end{equation*}
and such that $\int^\tau_0 R ds \leq k$ a.s.  Assuming that, for all stopping times $0\leq\tau_a <\tau_b<\tau$

\begin{displaymath}
  \mathbb{E}(\sup_{\tau_{a} < t <\tau_{b}} X + \int^{\tau_b}_{\tau_a} Yds) \leq
C_0 \left( \mathbb{E}( X(\tau_a) + \int^{\tau_b}_{\tau_a} (RX + Z)ds\right)
\end{displaymath}
where $C_0$ is a constant independent of the choice of $\tau_a$ and $\tau_b$, then
\begin{equation*}
\mathbb{E}\left( \sup_{0 < t < \tau} X + \int^\tau_0 Yds\right) \leq C\mathbb{E}\left( X(0) + \int^\tau_0 Zds\right),
\end{equation*}
where $C = C (C_0,T,K).$  We now just apply this lemma with $\tau = \tau_n, X =||U||^2, Y=|AU|^2, R=
c(1+|u|^2_{(2)}), Z = c (1+|F|^2)$ and the result follows.
\end{proof}

\subsection{Stopping time arguments}
\label{sec:stopp-time-argum}

We now implement the stopping time arguments that, applied in
combination with Lemmas~\ref{thm:AuxSystemEst} - \ref{thm:StrongEstU}, imply that $\xi = \infty$.

We define the stochastic processes
\begin{equation}\label{eq:normsofPossibleBlowup}
  \begin{split}
  X_1(t) &:= \sup_{0 \leq t' \leq t \wedge \xi} | \pd{z} \hat{u}|^2
                   + \int_0^{t \wedge \xi} \| \pd{z} \hat{u}\|^2 dt'\\
  X_2(t) &:= \sup_{0 \leq t' \leq t \wedge \xi} | \pd{x} \hat{u}|^2
                  + \int_0^{t \wedge \xi} \| \pd{x} \hat{u}\|^2 dt'\\
 X(t) &:=\sup_{0 \leq t' \leq t \wedge \xi} \| U\|^2
                + \int_0^{t \wedge \xi} |A U |^2 dt'\\
\end{split}
\end{equation}
and recall, with Lemmas \ref{thm:uzestimatesLemma} and \ref{thm:uxestimatesLemma} that $X_1(t)$ and
$X_2(t)$ are almost surely finite for all $t \geq 0$.  For $X(t),$ it follows from Lemma
\ref{thm:StrongEstU} that $X(t)$ is a.s. finite for every $t \in [0, \tau_{n}]$ where $\tau_n$ is
defined by \eqref{eq:controlBoundStrongFirstComp}.

We first aim to show that $\tau_n\uparrow\infty$ a.s. as $n\rightarrow\infty.$  Recalling that 
$u=\hat{u} + \check{u}$, we observe that $|u|_{(2)}^2\leq 2|\hat{u}|_{(2)} + 2|\check{u}|^2_{(2)}$ 
and infer, with Chebyshev's inequality, that for any $t > 0$,
\begin{displaymath}
\begin{split}
\mathbb{P} (\tau_n < t) &\leq \mathbb{P}\left(\int^{\xi\wedge t}_0|u|^2_{(2)} ds > n\right)\\
&\leq \mathbb{P}\left(\int^{\xi\wedge t}_0 |\hat{u}|^2_{(2)} ds >\frac{n}{2}\right) + \mathbb{P}
\left(\int^{\xi\wedge t}_0 |\check{u}|^2_{(2)} ds > \frac{n}{2}\right)\\
&\leq\mathbb{P}\left( X_1(t) + X_2(t) > cn\right) + \frac{c}{n}\mathbb{E}\int^t_0|\check{u}|^2_{(2)}ds.
\end{split}
\end{displaymath}
Thanks to \eqref{eq:checkStrongEstimate} this implies that
\begin{displaymath}
\lim_{n\rightarrow\infty} \mathbb{P}(\tau_n < t ) \leq \mathbb{P}(X_1(t) + X_2(t) = \infty)=0.
\end{displaymath}
Observing that the sequence $\tau_n$ is a.s. increasing, we have
\begin{displaymath}
\mathbb{P} \left(\lim_{n\rightarrow\infty}\tau_n < t\right) = \lim_{n\rightarrow\infty}\mathbb{P}
(\tau_n <  t) = 0,
\end{displaymath}
and hence $\tau_n\uparrow\infty$ a.s. as $n\rightarrow\infty.$

We now consider, for any $M > 0,$ the stopping time
\begin{displaymath}
\sigma_M = \inf \left\{ r \geq 0 :   X(r) > M\right\}
\end{displaymath}
and, in view of applying Proposition 5.1 below we want to evaluate $\mathbb{E} X (\tau_n\wedge\sigma_M\wedge t).$
To this end, we employ Lemma~\ref{thm:StrongEstU}  and infer that
\begin{displaymath}
  \sup_M \E X( \tau_n \wedge \sigma_M \wedge t)
   \leq K^5_n(t) < \infty.
\end{displaymath}
We finally conclude, by invoking Proposition~\ref{thm:stArg1},
that $X(t) < \infty$ for any $t >0$.  This implies
\begin{equation}\label{eq:finiteBlowUpCondOnStrongNorms}
  X(\xi(\omega)) < \infty
  \textrm{ for a.a. } \omega \in \{ \xi < \infty \}
\end{equation}
but since $(U,\xi)$ is a maximal strong solution (cf.
(\ref{eq:FiniteTimeBlowUp})), we perforce conclude that $\xi = \infty$
a.s.  The proof of Theorem~\ref{thm:MainReslt} is thus complete.

\section{Appendix I: An Abstract Stopping Time Result}
\label{sec:append-abstr-lemm}

We have made use of the following new result in the
final steps of the proof above of global existence.

\begin{Prop}\label{thm:stArg1}
  Fix $(\Omega,\mathcal{F}, \mathbb{P}, \{ \mathcal{F}_t\}_{t \geq
    0})$, a filtered probability space.  Let $X: \Omega \times
  [0,\infty) \rightarrow \mathbb{R}^+ \cup \{\infty \}$ be an
  increasing c\'{a}dl\'{a}g stochastic process and define
  \begin{displaymath}
    \sigma_M = \inf\{r \geq 0: X(r) \geq M \}.
  \end{displaymath}
  Suppose that there exists an increasing sequence of stopping times $\tau_n$ such
  that $\tau_n \uparrow \infty$ a.s. and such that for any fixed $n > 0$, $t >0$:
  \begin{displaymath}
    \kappa_{n,t} := \sup_M \mathbb{E} X(\tau_n\wedge\sigma_M\wedge t) < \infty.
  \end{displaymath}
  Then, for a set $\tilde{\Omega} \subset \Omega$ of full measure,
  \begin{equation}\label{eq:noBLOWUPBABY}
    X(t, \omega) < \infty, \quad \textrm{ for all } t \in [0, \infty), \omega \in \tilde{\Omega}.
  \end{equation}
\end{Prop}

\begin{proof}
  It is sufficient to show that $\lim_{M \rightarrow \infty}
  \mathbb{P}(\sigma_M < t) = 0$.  Indeed since
  \begin{displaymath}
    \{ X(t) < M \} \subseteq \{\sigma_M \geq t \}
  \end{displaymath}
  and since $\sigma_M$ is an increasing function of $M$, for any $M' > M$,
  \begin{displaymath}
    \{\sigma_{M} \geq t\} \subseteq \{ \sigma_{M'} \geq t\},
  \end{displaymath}
  we have that
  \begin{displaymath}
    \begin{split}
    \mathbb{P} ( X(t) < \infty) &=
      \mathbb{P} ( \cup_{M > 0} \{X(t) < M\})\\
    &\leq
      \mathbb{P} ( \cup_{M > 0} \{\sigma_M \geq t \})\\
    &=
     \lim_{M \rightarrow \infty}  \mathbb{P} ( \sigma_M \geq t )\\
    &=
     \lim_{M \rightarrow \infty} (1 - \mathbb{P} ( \sigma_M < t )).\\
    \end{split}
  \end{displaymath}

  Give any $M,n$, observe that since $X$ is right continuous and increasing,
    \begin{displaymath}
    \begin{split}
      \{\sigma_M < t, \tau_n \geq t\}
      &= \{X(\sigma_{M} \wedge t) \geq M, \sigma_{M} < t, \tau_n \geq t \}\\
      &\subseteq \{X(\sigma_{M} \wedge t) \geq M, \tau_n \geq t \}\\
      &\subseteq \{X(\sigma_{M} \wedge \tau_n \wedge t) \geq M\},\\
    \end{split}
  \end{displaymath}
  and therefore
  \begin{displaymath}
    \begin{split}
      \mathbb{P} (\sigma_M < t)
         & \leq \mathbb{P} (\sigma_M < t, \tau_n \geq t) +
                \mathbb{P} (\tau_n < t)\\
         & \leq \mathbb{P}(X(\sigma_M \wedge \tau_n \wedge t) \geq M) +
                \mathbb{P} (\tau_n < t)\\
         & \leq \frac{1}{M} \mathbb{E}(X(\sigma_M \wedge \tau_n \wedge t)) +
                \mathbb{P} (\tau_n < t)\\
         & \leq \frac{\kappa_{n,t}}{M} +
                \mathbb{P} (\tau_n < t).\\
    \end{split}
  \end{displaymath}
 Thus, for any fixed $n$ and $t$
  \begin{displaymath}
    \lim_{M \rightarrow \infty}\mathbb{P}(\sigma_M < t) \leq \mathbb{P}(\tau_n < t).
  \end{displaymath}
 However, given the assumptions on $\tau_n$, we have that
  \begin{displaymath}
    \lim_{n \rightarrow \infty}\mathbb{P}(\tau_n < t) =0,
  \end{displaymath}
  which shows that $X(t,\omega)<\infty$ a.s. for $\omega\in\Omega.$  To determine the set $\tilde{\Omega}$ in
  \ref{eq:noBLOWUPBABY} and complete the proof, we observe that $X$ is an increasing function of $t$ and call, for each
  $j\in\mathbb{N},\Omega_j$ the set of full measure such that $X(j,\omega) <\infty,\enspace\forall\omega\in\Omega_j.$
Then $X(t,\omega) <\infty$ for every $t,0\leq t\leq j,$ and we can take for $\tilde{\Omega}$, the intersection
$\cap_{j\geq 1}\Omega_j$ which is a set of full measure as well.
\end{proof}

\section*{Acknowledgments}
This work was partially supported by the National Science Foundation under the grants
DMS-0604235, DMS-0906440, and DMS- 1004638
and by the Research Fund of Indiana University.

\bibliographystyle{plain}
\bibliography{ref_teressa}

\end{document}